\documentclass[11pt]{article}

\usepackage{amsmath,amssymb,amsthm}
\usepackage{bm}
\usepackage{graphicx}
\usepackage{graphicx}  
\usepackage{multirow} 
\usepackage{booktabs}  
\usepackage{hyperref}
\usepackage{geometry}
\usepackage{subcaption}
\usepackage[round,authoryear]{natbib}

\usepackage{booktabs}
\usepackage{subcaption} 
\usepackage{subcaption}
\usepackage{float}
\usepackage{enumitem} 
\usepackage{cancel}
\usepackage{pgffor}

\renewcommand{\theequation}{\thesection.\arabic{equation}}

\newtheorem{theorem}{Theorem}
\newtheorem{prop}{Proposition}
\newtheorem{corollary}{Corollary}
\newtheorem{remark}{Remark}

\title{Gamma-Based Expansion for the First-Passage Time Distribution\\
of Stochastic Logistic Models with Harvesting \footnote{This preprint has not undergone peer review or any post-submission improvements or corrections. The Version of Record of this article is published in Statistical Papers, and is available online at https://doi.org/10.1007/s00362-026-01949-1}}

\author{
Simone Catanzaro \and Elvira Di Nardo
\thanks{Dipartimento di Matematica, Univ. Torino, Italy, elvira.dinardo@unito.it}}

\date{}

\begin{document}
\maketitle

\begin{abstract}
 The first passage time problem is considered for stochastic logistic growth model with constant harvesting and multiplicative environmental noise. 
Explicit expressions for the moments and cumulants of both upcrossing and downcrossing FPTs in the presence of constant thresholds are obtained through a power-series expansion of the Laplace transform. Then a  closed-form representation of the FPT density is recovered via an orthogonal Laguerre--Gamma expansion .

This representation is used to numerically evaluate FPT densities, with the truncation order controlling the trade-off between accuracy and stability. Numerical experiments based on Monte Carlo simulations confirm the high accuracy of the method in regimes of moderate dispersion and highlight its limitations when higher-order moments grow rapidly. Application to fisheries management models shows that the method remains effective
even for large-scale population. Finally, the approximated density is satisfactory used to estimate some parameters of the model. 
\end{abstract}

\section{Introduction}

In a random environment, the growth dynamics of a population subject to harvesting
can be described by the stochastic differential equation  
\begin{equation}
\label{eq:SDE}
{\rm d}X(t) = f(X(t))\,X(t)\,{\rm d}t - q E X(t)\,{\rm d}t + \sigma X(t)\,{\rm d}W(t),
\qquad X(0)=x_0>0,
\end{equation}
where $X(t)$ is the population size at time $t,$
$W(t)$ is a standard Wiener process, $\sigma>0$ measures environmental variability, 
and $q E X(t)$ represents harvesting under a constant policy with $q>0$ the catchability coefficient and $E>0$ the constant harvesting effort.
The population size $x_0$ at time $t=0$  is assumed to be known. The function $f(X(t))$ represents the average \emph{per capita} natural growth
rate of the population. 
 Stochastic growth dynamics \eqref{eq:SDE} are widely employed in ecology and resource management to quantify extinction risk, persistence
properties, and recovery times under environmental uncertainty
\citep{Allen2010,Renshaw2011,Beddington1977,Lande1993}.

Suitable regularity conditions on $f$ ensure that \eqref{eq:SDE} admits a unique strong solution (up to a possible explosion time) \citep{Mao2007,KarlinTaylor1981}. These are satisfied for the stochastic logistic model
\begin{equation}
\label{eq:logistic_f}
f(X(t)) = r\left(1 - \frac{X(t)}{K}\right)\! ,
\end{equation}
where $r>0$ is the intrinsic growth rate and $K>0$ the carrying capacity. In the logistic model~\eqref{eq:logistic_f}, the condition for
non-extinction can be written as
$E < \frac{1}{q}\left(r - \frac{\sigma^{2}}{2}\right)\! .$
The solution of \eqref{eq:SDE} with $f$ given in
\eqref{eq:logistic_f} is a one-dimensional time-homogeneous diffusion process
with drift and diffusion coefficient respectively
\[
a(x_0) = r x_0\left(1 - \frac{x_0}{K}\right) - q E x_0 \qquad b(x_0) = \sigma^{2} x_0^{2},
\]
see, e.g., \cite{Tuckwell1974,Braumann1999} for closed-form representations and
further properties.
The state space is $(0,\infty)$, with $0$ absorbing (extinction) and $\infty$ non-attractive. The parameter
\begin{equation}
\rho = \frac{2(r - qE)}{\sigma^{2}} - 1
\label{rho}
\end{equation}
measures the balance between growth and noise near extinction. If $\rho>0$, the process admits a stationary Gamma distribution with shape $\rho$ and rate $\theta = 2r/(K\sigma^2)$ \citep{Beddington1977,Braumann1999}, and all interior states are reached almost surely in finite time. This motivates the study of first passage times (FPTs) to thresholds  confined to the state space. 
The FPT random variable is the random time at which $X(t)$, starting from $x_0$, first reaches a prescribed threshold, \cite{redner2001guide}. Depending on the application, the threshold represents a biological or management level, and both upcrossing and downcrossing problems are of interest.

In general, the  FPT probability density function (pdf) provides more information than summary statistics, as it characterizes the risk, timing, and variability of critical events. However, closed-form expressions are rarely available and, when they exist, are often too involved for practical use.
Various approaches have been proposed, whose applicability depends on the structure of the underlying process (see \cite{di2023approximating} and references therein). Among them, the Laplace transform of the FPT pdf is widely used, as it often admits a closed-form even when the pdf does not.  However, recovering the FPT pdf from its Laplace transform generally requires inversion of special functions, leading to expressions that are quite often not analytically tractable. As a result, most methods rely on numerical inversion or Monte Carlo simulation, which may be computationally expensive or yield inefficient results.

The \texttt{R} package \texttt{fptdApprox} \citep{fptdApproxManual} approximates FPT pdfs for one-dimensional diffusions using a numerical solution of a  second-kind Volterra integral equation. Its applicability, however, requires explicit knowledge of the transition pdf and distribution of the underlying process. For logistic models with harvesting, a closed-form expression for the transition pdf is available under suitable conditions \citep{Otunuga2021}, and applies in particular when $\rho>0$. This representation involves Laguerre polynomials and an integral of Whittaker functions \citep{Olver2010}, which makes its practical implementation within \texttt{fptdApprox} cumbersome. This motivates the development of alternative approaches based on the Laplace transform and its analytic properties.

Closed-form Laplace transforms of the FPT pdf for the stochastic logistic model without harvesting were derived in \cite{GietValloisWantz2015} and extend to constant harvesting via reparametrization. In this setting, \cite{BritesBraumann2022} obtained expressions for the FPT mean and variance and approximated the FPT pdf by numerical inversion of the Laplace transform. While effective, this approach is computationally demanding and mainly captures tail behavior. FPTs for stochastic logistic diffusions have also been studied in \citet{Otunuga2025}, who analyze moments and distributions under suitable thresholds, but do not provide analytical approximations for constant thresholds.

An alternative strategy  relies on Laguerre–Gamma expansions, successfully used to approximate the (upcrossing) FPT pdf for the Feller diffusion \citep{DOM2024} and for the inhomogeneous geometric Brownian motion \citep{DiNardoDOnoFrio2021}.
The aim of this paper is to extend this approach to stochastic logistic models with harvesting, in oder to recover an approximation of the upcrossing and downcrossing FPT pdfs.
We focus the attention on constant thresholds.

Our contribution is both methodological and numerical.
Starting from the explicit form of the Laplace transform of the
FPT pdf, we exploit a suitable power series expansion to derive
closed-form expressions for FPT moments and cumulants.
These quantities are then used to construct a Laguerre-Gamma orthogonal
expansion of the FPT pdf.
The theoretical results are complemented by a numerical study aimed
at assessing the accuracy, robustness, and potential limitations of the
proposed approach.
To this end, we compare the approximated densities with reference results
obtained from Monte Carlo simulations.
Finally, we illustrate the applicability of the method to fisheries
management models and show how the approximated density can be effectively
employed within a maximum likelihood framework for parameter estimation.

\section{FPT moments}

When the threshold lies above the initial state $x_0$,  the upcrossing
FPT is defined as 
\begin{equation}
\label{eq:FPT_up}
T_U := \inf\{t \ge 0 : X(t) > U\}, \qquad 0 < x_0 < U < +\infty,
\end{equation}
whereas, when the threshold lies below $x_0$, the
downcrossing FPT is defined as
\begin{equation}
\label{eq:FPT_down}
T_L := \inf\{t \ge 0 : X(t) < L\}, \qquad 0 < L < x_0 < +\infty.
\end{equation}
Denote by
\[
\tilde{g}_U(\lambda)=\mathbb{E}_{x_0}\!\left[e^{-\lambda T_U}\right]
\quad \text{and} \quad
\tilde{g}_L(\lambda)=\mathbb{E}_{x_0}\!\left[e^{-\lambda T_L}\right]
\]
the Laplace transforms of the FPT r.v.'s $T_U$ and
$T_L$, respectively. For the stochastic logistic model with harvesting, explicit expressions of $\tilde{g}_U(\lambda)$ and $\tilde{g}_L(\lambda)$ are 
given in \cite{GietValloisWantz2015}. Set
\[
r_1 = r - qE,
\qquad
K_1 = K\left(1 - \frac{qE}{r}\right),
\]
and rewrite
\eqref{eq:SDE} with $f$ given by~\eqref{eq:logistic_f} as
\begin{equation}
\label{eq:SDE_effective}
dX(t) = r_1 X(t)\left(1 - \frac{X(t)}{K_1}\right)\,dt + \sigma X(t)\,dW(t),
\qquad X(0)=x_0.
\end{equation}
Then the Laplace transforms are  
\begin{equation}
\label{eq:LT_TU}
\tilde{g}_U(\lambda)
=
\left(\frac{x_0}{U}\right)^{\sqrt{\frac{2\lambda}{\sigma^2}+u^2}+u}
\frac{
\Phi\!\left(
\sqrt{\frac{2\lambda}{\sigma^2}+u^2}+u,\;
1+2\sqrt{\frac{2\lambda}{\sigma^2}+u^2},\;
v x_0
\right)
}{
\Phi\!\left(
\sqrt{\frac{2\lambda}{\sigma^2}+u^2}+u,\;
1+2\sqrt{\frac{2\lambda}{\sigma^2}+u^2},\;
v U
\right)
}
\end{equation}
and
\begin{equation}
\label{eq:LT_TL}
\tilde{g}_L(\lambda)
=
\left(\frac{x_0}{L}\right)^{\sqrt{\frac{2\lambda}{\sigma^2}+u^2}+u}
\frac{
\Psi\!\left(
\sqrt{\frac{2\lambda}{\sigma^2}+u^2}+u,\;
1+2\sqrt{\frac{2\lambda}{\sigma^2}+u^2},\;
v x_0
\right)
}{
\Psi\!\left(
\sqrt{\frac{2\lambda}{\sigma^2}+u^2}+u,\;
1+2\sqrt{\frac{2\lambda}{\sigma^2}+u^2},\;
v L
\right)
}
\end{equation}
respectively, with \begin{equation}
u = \frac{1}{2}\left(1 - \frac{2r_1}{\sigma^{2}}\right),
\qquad
v = \frac{2r_1}{K_1 \sigma^{2}}
\label{parameters} 
\end{equation}
and $\Phi$ and $\Psi$ denoting confluent hypergeometric functions, also known as the Kummer and Tricomi functions  respectively \footnote{These functions are denoted by $M$ and $U$ in
\cite{AbramowitzStegun1964}respectively.}. Note that  the upcrossing Laplace transform  \eqref{eq:LT_TU} can be written as 
\begin{equation}
\tilde g_U(\lambda)
=
\frac{G_U(\lambda,x_0)}{G_U(\lambda,U)}, \qquad \hbox{\rm with} \,\,\,\, G_U(\lambda,y)
:=
y^{\,u[1-s(\lambda)]}\,
\Phi\!\big(u[1-s(\lambda)],\,1-2u\,s(\lambda),\,v y\big),
\label{GU}
\end{equation}
and  the downcrossing Laplace transform   \eqref{eq:LT_TL} as
\begin{equation}
\tilde g_L(\lambda)
=
\frac{G_L(\lambda,x_0)}{G_L(\lambda,L)}, \qquad \hbox{\rm with} \,\,\,\, G_L(\lambda,y)
:=
y^{\,u[1-s(\lambda)]}\,
\Psi\!\big(u[1-s(\lambda)],\,1-2u\,s(\lambda),\,v y\big)
\label{GL}
\end{equation}
where $
a := 2/(\sigma^{2}u^{2}),$ and 
$s(\lambda) := \sqrt{1 + a\lambda}.$ 
\begin{remark}{\rm 
\label{rem:LT_normalization}
If $\rho > 0$ ($u <0$), all interior thresholds are reached almost
surely in finite time, that is  $\mathbb{P}_{x_0}(T_U<\infty)=\mathbb{P}_{x_0}(T_L<\infty)=1$ or equivalently $\tilde g_U(0)=\tilde g_L(0)=1.$
Moreover, if the initial condition coincides with the threshold, that is,
$x_0=U$ for the upcrossing case or $x_0=L$ for the downcrossing case, then
$T_U=0$ or $T_L=0$ almost surely and
$\tilde g_U(\lambda)=1$ or $\tilde g_L(\lambda)=1,$ for all $\lambda \ge 0.$
Finally, for $\lambda\ge0$ and for any $x_0 \in (0,\infty)$ and different from the threshold, the Laplace transforms
\eqref{eq:LT_TU} and \eqref{eq:LT_TL} are well defined \citep{GietValloisWantz2015}.}
\end{remark}

\subsection{Power series expansion}

Let $T$ be a non-negative r.v. with Laplace transform $\tilde g(\lambda)=\mathbb{E}_{x_0}\!\left[e^{-\lambda T}\right].$ If $\tilde g(\lambda)$ is analytic in a neighbourhood of $\lambda=0$,  it admits a convergent Taylor expansion 
\begin{equation}
\label{eq:LT_series}
\tilde g(\lambda)
=
\sum_{k \geq 0}  \tilde g_k\,\frac{\lambda^k}{k!},
\qquad |\lambda|<R,
\end{equation}
for some $R>0.$ Thus $T$ has moments of all orders \citep{AbateWhitt1996} such that
\begin{equation}
\label{eq:coeffs_moments}
\mathbb{E}_{x_0}\!\left[T^k\right]
=
(-1)^k\,\tilde g_k,
\qquad k \ge 0.
\end{equation}
For $\rho > 0$ ($u < 0$), the following propositions show that the functions $G_U(\lambda,\cdot)$ and $G_L(\lambda,\cdot)$ in \eqref{GU} and \eqref{GL} admit power series expansions in $\lambda$ around the origin. As a consequence, the power series expansions of $\tilde g_U(\lambda)$ and $\tilde g_L(\lambda)$ can be derived from those of $G_U(\lambda,\cdot)$ and $G_L(\lambda,\cdot)$, and all the FPT moments can be recovered through \eqref{eq:coeffs_moments}. 
In the following  $(x)_m = x (x-1) \cdots (x-m+1)$ denotes the falling factorial and 
$\Theta = t \frac{\rm d}{{\rm d} t}$ is the Euler operator\footnote{For the Euler operator, see \cite{charalambides2002enumerative, jordan1949calculus}.} such that $\Theta^k t^n = k^n t^n, k,n \geq 1.$ 
\begin{theorem}[Upcrossing]
\label{UP}
In a suitable neighborhood of
$\lambda=0,$ we have
\begin{equation}
G_U(\lambda,y)=  \sum_{m \geq 0} t_m(y) \frac{ \lambda^m}{m!}
\label{GUps}
\end{equation}
where 
\begin{equation}
t_m(y) = \sum_{k=0}^m \binom{m}{k} l_{m-k}(y) q_k, \,\, q_k= - u \log y \sum_{j=1}^k \binom{k-1}{j-1} \left( \frac{1}{2} \right)_{\!\!j} a^j q_{k-j}, \,\, l_k(y)= a^k\left[\sum_{n \geq 1} M_{n,k} \frac{(vy)^n}{n!}\right]
\label{0coeff}
\end{equation}
with $q_0=l_0(y)=1,$ 
\begin{eqnarray}
M_{n,m} & = & \frac{1}{\langle 1- 2u \rangle_n} \left(\tilde{\Lambda}_{n,m} - \sum_{k=1}^m \binom{m}{k} M_{n,m-k}\, \Lambda_{n,k}
\right) \quad \hbox{$n,m \geq 1$}
\label{ratioM} \\
{\Lambda}_{n,k} & = &  \left. \left( \frac{\Theta}{2} \right)_k \langle 1 - 2u t \rangle_n \right|_{t=1} \quad \hbox{\rm and} \quad \tilde{\Lambda}_{n,m} = \left. \left( \frac{\Theta}{2} \right)_m \langle u(1-t) \rangle_n \right|_{t=1}.
\label{firstcoeff}
\end{eqnarray}
\end{theorem}
\noindent
For the proof see Appendix B.
\begin{theorem}[Downcrossing]
\label{down}
In a suitable neighborhood of
$\lambda=0,$ and as
\(vy\to\infty\), the function \(G_L(\lambda,y)\) admits the
Poincar\'e asymptotic expansion
\begin{equation}
G_L(\lambda,y)
\sim
\frac{\tilde G_L(\lambda,y)}
     {v^{u[1-s(\lambda)]}},
\qquad
\tilde G_L(\lambda,y)
=
1+\sum_{m\geq1}\overline l_m(y)\frac{\lambda^m}{m!},
\label{raprGL}
\end{equation}
where the series defining \(\tilde G_L(\lambda,y)\) is understood
formally, with
\begin{equation}
\overline l_m(y)
=
a^m\sum_{n\geq1}
\frac{(-1)^n}{(vy)^n n!}\,
\bar{M}_{n,m},
\qquad
\bar{M}_{n,m}
=
\sum_{k=0}^{m}
\binom{m}{k}
\tilde\Lambda_{n,k}\,
\bar{\Lambda}_{n,m-k},
\label{lbar}
\end{equation}
and
\begin{equation}
\bar{\Lambda}_{n,m}
=
\left.
\left(\frac{\Theta}{2}\right)_m
\langle u(1+t)\rangle_n
\right|_{t=1},
\label{IIcoeffdowncomp}
\end{equation}
and \(\{\tilde\Lambda_{n,k}\}\) are given in~\eqref{firstcoeff}.
\end{theorem}
\noindent
\begin{remark}
\label{remark2}
{\rm The representation of the Tricomi function used in the proof of
Theorem~\ref{down} is its Poincar\'e asymptotic expansion, expressed through
the formal \({}_2F_0\) series. Hence, all subsequent downcrossing
formulas are understood asymptotically and evaluated by finite
truncation for sufficiently large \(vx_0\) and \(vL\). Such parameter
regimes occur in applications considered in the literature; see,
for example, \cite{BritesBraumann2022}.}
\end{remark}
\noindent
For the proof see Appendix B.
The 
quantities $\Lambda_{n,k}$, $\tilde{\Lambda}_{n,m}$, and $\bar{\Lambda}_{n,m}$ in \eqref{firstcoeff} and \eqref{IIcoeffdowncomp}
can be computed by making explicit
the action of the Euler operator in \eqref{firstcoeff}--\eqref{IIcoeffdowncomp}, as 
the following proposition shows. These results are useful in the numerical applications.
\begin{prop}
For $n,m,k \ge 1$, the quantities $\Lambda_{n,k}$, $\tilde{\Lambda}_{n,m}$, and $\bar{\Lambda}_{n,m}$ in \eqref{firstcoeff} and \eqref{IIcoeffdowncomp} admit the following explicit representations:
\begin{align}
\Lambda_{n,k}
&= \sum_{j=1}^n (-2u)^j  
\left[ \begin{matrix} n+1 \\ j+1 \end{matrix} \right] \left( \frac{j}{2} \right)_k,
\label{esplicita} \\
\tilde{\Lambda}_{n,m}
&= \sum_{k=1}^n A_{n,k} \left( \frac{k}{2} \right)_m,
\qquad
A_{n,k} = (-1)^k \sum_{j=k}^n \left[ \begin{matrix} n \\ j \end{matrix} \right] \binom{j}{k} u^j,
\label{esplicita1} \\
\bar{\Lambda}_{n,m}
&= \sum_{k=1}^n D_{n,k} \left( \frac{k}{2} \right)_m,
\qquad
D_{n,k} = \sum_{j=k}^n \left[ \begin{matrix} n \\ j \end{matrix} \right] \binom{j}{k} u^j.
\label{esplicita2}
\end{align}
Here, $\left[ \begin{matrix} n \\ j \end{matrix} \right]$ denote the unsigned Stirling numbers of the first kind.
\end{prop}
\begin{proof}
Using the well-known identity
$
\langle x+1\rangle_n
=
\sum_{j=0}^{n}
\Big[\!\begin{matrix} n+1 \\ j+1 \end{matrix}\!\Big] x^j
$
we have
\begin{equation}
\langle 1-2u t\rangle_n 
=
\sum_{j=0}^{n}
\Big[\!\begin{matrix} n+1 \\ j+1 \end{matrix}\!\Big]
(-2u)^j t^j.
\label{expG}
\end{equation}
Since 
$\left(\frac{\Theta}{2}\right)_k (t^j)
=
\left(\frac{j}{2}\right)_k t^j,
$ with $\Theta = t \,\frac{d}{dt}$ the Euler operator, applying $\left(\frac{\Theta}{2}\right)_k$ to \eqref{expG} and evaluating at $t=1$ yields the expression for $\Lambda_{n,k}$ in \eqref{esplicita}.
Next, the identity
$
\langle x\rangle_n
=
\sum_{j=0}^n
\left[ \begin{matrix} n \\ j \end{matrix} \right] x^j
$
implies
\begin{equation}
\langle u(1-t)\rangle_n
=
\sum_{j=0}^n
\left[ \begin{matrix} n \\ j \end{matrix} \right] u^j (1-t)^j
=
\sum_{k=0}^n \left(
(-1)^k
\sum_{j=k}^n
\left[ \begin{matrix} n \\ j \end{matrix} \right]
\binom{j}{k}
u^j \right) t^k,
\label{expH}
\end{equation}
after expanding $(1-t)^j$ and rearranging the sums. Applying the operator $\left(\frac{\Theta}{2}\right)_m$ gives
\begin{equation}
\left(\frac{\Theta}{2}\right)_m \langle u(1-t)\rangle_n
=
\sum_{k=0}^n
\left(
(-1)^k
\sum_{j=k}^n
\left[ \begin{matrix} n \\ j \end{matrix} \right]
\binom{j}{k}
u^j \right)
\left(\frac{k}{2}\right)_m
t^k.
\label{expH1}
\end{equation}
Evaluating \eqref{expH1} at $t=1$ yields the expression for $\tilde{\Lambda}_{n,m}$ in  \eqref{esplicita1}. The expression for 
$\bar{\Lambda}_{n,m}$ in
\eqref{esplicita2} follows analogously by replacing $\langle u(1-t)\rangle_n$ with $\langle u(1+t)\rangle_n$.
\end{proof}
The following corollary  provides recursive formulas for the coefficients in the expansions \eqref{eq:LT_series} of $\tilde{g}_U(\lambda)$ and $\tilde{g}_L(\lambda)$,
from which the corresponding FPT moments follow.
\begin{corollary} \label{cor1}
Let $S \in \{U,L\}$, with $S=U$ if $x_0 < U$ and $S=L$ if $x_0 > L$. Then, for $m \ge 0$, the $m$-th moment of the FPT rv $T$ is 
$\mathbb{E}[T^m] = (-1)^m \tilde g_m,$ 
where $\tilde g_0 = 1$ and
\begin{equation}
\label{eq:recursion_gm}
\tilde g_m = s_m(x_0) - \sum_{k=1}^m \binom{m}{k} \, s_k(S)\, \tilde g_{m-k} 
\,\,\hbox{with} \,\,\,
s_m(y)= \left\{  
\begin{array}{cl}
t_m(y) & \hbox{if $S=U,$} \\
\bar{l}_m(y) & \hbox{if $S = L.$}
\end{array} \right.
\end{equation}
\end{corollary} 
\begin{proof}
Let us first consider the upcrossing case $x_0 < S=U.$ By equating the coefficients of equal powers of $\lambda$ in $G_U(\lambda, U) \tilde{g}_U(\lambda)$ and $G_U(\lambda, x_0),$ and by recalling that $t_0(x_0)=t_0(U)=1,$ the upcrossing FPT Laplace transform
$\tilde g_U(\lambda)$ admits the power- series expansion \eqref{eq:LT_series} with coefficients 
$\{\tilde{g}_m\}$ given in \eqref{eq:recursion_gm} and ${s_m(y)}={t_m(y)}$ given in \eqref{0coeff}. For the downcrossing case $x_0 > S = L$, the result follows by observing that, upon substituting $G_L(\lambda,x_0)$ as given in  \eqref{raprGL} into the ratio \eqref{GL}, the downcrossing FPT Laplace transform can be written as
\begin{equation}
\tilde g_L(\lambda)
= \frac{\tilde{G}_L(\lambda,x_0)}{\tilde{G}_L(\lambda,L)}.
\label{GLII}
\end{equation}
Therefore, by repeating the same reasoning as in the upcrossing case, the coefficients $\{\tilde{g}_m\}$ of the power series expansion \eqref{eq:LT_series} of $\tilde g_L(\lambda)$ satisfy the recursion \eqref{eq:recursion_gm}, with $s_m(y)=\bar{l}_m(y)$ given in \eqref{lbar}.
\end{proof}

Recursion~\eqref{eq:recursion_gm} is particularly well suited for numerical implementation. By computing the reciprocals of $G_U(\lambda,U)$ in \eqref{GUps} (for $x_0 < U$) and $\tilde{G}_L(\lambda,L)$ in \eqref{raprGL} (for $x_0 > L$) via \eqref{reciprocochiuso}, and exploiting the Cauchy product of power series, we derive the following closed-form expressions for the FPT moments.
\begin{corollary}Under the same assumptions as in  Corollary \ref{cor1}, 
we have
\begin{equation}
\mathbb{E}[T^m] = (-1)^m \bigg[s_m(x_0) +\sum_{j=1}^m \binom{m}{j} s_{m-j}(x_0)
\left( \sum_{k=1}^j (-1)_k\, B_{j,k}\bigl(s_1(S), \ldots, s_{j-k+1}(S)\bigr) \bigg) \right],
\label{momFPT}
\end{equation}
where $\{B_{j,k}\}$ denote the partial exponential Bell polynomials defined in~\eqref{polinc}.
\end{corollary} 


\section{FPT cumulants} 
From \eqref{eq:LT_series} the logarithm of the Laplace transform is
well defined as a formal power series and the cumulants of $T$ can be recovered
accordingly.
In particular, they are given by
\begin{equation}
\label{eq:coeffs_cumulants}
c_k(T)
=
(-1)^k\,\tilde c_k,
\qquad k \ge 1,
\end{equation}
where $\{\tilde c_k\}_{k\ge1}$ are defined through the expansion
\begin{equation}
\label{eq:formal_log}
\log \tilde g(z)
=
\sum_{k \geq 1}  \tilde c_k\,\frac{z^k}{k!}.
\end{equation}

While the relations between moments and cumulants are well known -either in recursive form or via Bell polynomials \citep{DOM2024} -the ratio structure of the Laplace transforms considered here yields an explicit formulation.
Indeed, taking logarithms in \eqref{GU} and \eqref{GLII}, we have
\begin{equation}
\log \tilde g_U(\lambda)
=
\log G_U(\lambda,x_0) - \log G_U(\lambda,U),
\qquad
\log \tilde g_L(\lambda)
=
\log \tilde{G}_L(\lambda,x_0) - \log \tilde{G}_L(\lambda,L).
\label{logdiff}
\end{equation}
Accordingly, the coefficients in the power series expansions of
$\log \tilde g_U(\lambda)$ and $\log \tilde g_L(\lambda)$ are given by the
differences of the corresponding coefficients in the expansions of
$\log G_U(\lambda,\cdot)$ and $\log \tilde{G}_L(\lambda,\cdot)$, evaluated at $x_0$ and at the threshold respectively. This observation is used in the proof of the following theorem.
\begin{theorem}[Upcrossing]
For $k \ge 1$, the $k$-th upcrossing FPT  cumulant is given by
\begin{equation}
\label{eq:cumulantU}
c_k(T_U)
= (-1)^k
\left[
u \log\!\left(\frac{U}{x_0}\right)
\left(\frac{1}{2}\right)_k \left( \frac{2}{\sigma^2 u^2} \right)^k
+ \tilde c_k(x_0) - \tilde c_k(U)
\right],
\end{equation}
where $u$ is defined in~\eqref{parameters}, and
$
\tilde c_k(y) = L_k[l_1(y), \ldots, l_k(y)],
$
with $\{L_k\}$ the logarithmic polynomials  in~\eqref{logpoly}, and $\{l_k(y)\}$ given in~\eqref{0coeff}.
\end{theorem}
\begin{proof}
Take the logarithm of $G_U(\lambda,y)$ in \eqref{GU}. 
By expanding $1-s(\lambda) = 1-(1 + a \lambda)^{1/2}, $ 
we recover
\begin{equation} 
\log y^{u(1-s(\lambda))} = - u \log y \sum_{k \ge 1} \left( 
\frac{1}{2} \right)_k \left( \frac{2}{\sigma^2 u^2} \right)^k \frac{\lambda^k}{k!}.
\label{exp1}
\end{equation}
From \eqref{ap:phiexpansion}, \eqref{expT} and \eqref{logexpansion} we have 
\begin{equation} 
\log \Phi\!\big(u[1-s(\lambda)],\,1-2u\,s(\lambda),\,v y\big) = \sum_{k\ge 1} L_k[l_1(y), \dots, l_k(y)]\,
\frac{\lambda^k}{k!}
\end{equation}
and the result follows from 
the first equation in \eqref{logdiff}. 
\end{proof}
By analogous arguments applied to $\log \tilde{G}_L(\lambda,y)$ in \eqref{GLII}, the downcrossing FPT cumulants can likewise be expressed in terms of logarithmic polynomials, as stated in the following theorem.
\begin{theorem}[Downcrossing]
For $k \ge 1$, the $k$-th downcrossing FPT  cumulant is given by
\begin{equation}
\label{eq:cumulantD}
c_k(T_L)
= (-1)^k
\left[\bar c_k(x_0) - \bar c_k(L)
\right],
\end{equation}
where $\bar c_k(y)_=L_k[\bar{l}_1(y), \ldots, \bar{l}_k(y)]$ with $\{L_k\}$ the logarithmic polynomials introduced in~\eqref{logpoly}, and $\{\bar{l}_k(y)\}$ given in \eqref{lbar}.
\end{theorem}
In particular, since 
$c_1(T_U) = {\mathbb E}[T_U]$ and $c_1(T_L) = {\mathbb E}[T_L],$ the upcrossing and downcrossing FPT means are respectively 
\begin{eqnarray*}
{\mathbb E} [T_U] & = & \frac{1}{u \sigma^2}
\log \frac{x_0}{U} + \frac{2}{\sigma^2 u^2} \sum_{n \geq 0} \frac{\tilde{\Lambda}_{n,1} \Lambda_{n,0} - \tilde{\Lambda}_{n,0} \Lambda_{n,1}}{\Lambda^2_{n,0}} \frac{[(vU)^n - (vx_0)^n]}{n!} \\
\quad {\mathbb E} [T_L] & = &  \frac{2}{\sigma^2 u^2} \sum_{n \geq 1} (-1)^n \frac{\tilde{\Lambda}_{n,0} \bar \Lambda_{n,1} + \tilde{\Lambda}_{n,1} \bar \Lambda_{n,0} }{n!}\bigg[\frac{1}{(vL)^n} - \frac{1}{(vx_0)^n}\bigg].
\end{eqnarray*}
We also have $c_2(T_U)={\rm Var}(T_U)$ and 
$c_2(T_L)={\rm Var}(T_L).$ Therefore
the upcrossing and downcrossing FPT variances are respectively
\begin{eqnarray*}
{\rm Var}(T_U) \!\!\!\! & = &  \!\!\!\!
\frac{\log(x_0/U)}{\sigma^4 u^3}
+ \frac{4}{\sigma^4 u^4} \left\{ \sum_{n \geq 0}
\frac{
\tilde{\Lambda}_{n,2}\Lambda_{n,0}^2
-2\tilde{\Lambda}_{n,1}\Lambda_{n,1}\Lambda_{n,0}
-\tilde{\Lambda}_{n,0}\Lambda_{n,2}\Lambda_{n,0}
+2\tilde{\Lambda}_{n,0}\Lambda_{n,1}^2
}{\Lambda_{n,0}^3}
\bigg( \frac{(vx_0)^n}{n!} \right. \\
& - & \!\!\!\! \frac{(vU)^n}{n!} \bigg) + \left. \left( \sum_{n \geq 0} \frac{\tilde{\Lambda}_{n,1} \Lambda_{n,0} - \tilde{\Lambda}_{n,0} \Lambda_{n,1}}{\Lambda^2_{n,0}} \frac{(vU)^n}{n!}\right)^2 - \left( \sum_{n \geq 0} \frac{\tilde{\Lambda}_{n,1} \Lambda_{n,0} - \tilde{\Lambda}_{n,0} \Lambda_{n,1}}{\Lambda^2_{n,0}} \frac{(vx_0)^n}{n!}\right)^2 \right\} \\
 {\rm Var}(T_L) \!\!\!\! & = & \!\!\!\!
 \frac{4}{\sigma^4 u^4} \left\{ \sum_{n\ge 1}\frac{(-1)^n}{n!}
\left(
\frac{1}{(v x_0)^n}
-
\frac{1}{(v L)^n}
\right)
\left(
\tilde\Lambda_{n,0}\bar\Lambda_{n,2}
+2\,\tilde\Lambda_{n,1}\bar\Lambda_{n,1}
+\tilde\Lambda_{n,2}\bar\Lambda_{n,0}
\right) \right. \\
& + & \left.
 \bigg(\sum_{n \geq 1}\frac{\tilde{\Lambda}_{n,0} \bar \Lambda_{n,1} + \tilde{\Lambda}_{n,1} \bar \Lambda_{n,0} }{n!}\frac{(-1)^n}{(vL)^n} \bigg)^2 - 
\bigg( \sum_{n \geq 1}  \frac{\tilde{\Lambda}_{n,0} \bar \Lambda_{n,1} + \tilde{\Lambda}_{n,1} \bar \Lambda_{n,0} }{n!}\frac{(-1)^n}{(vx_0)^n} \bigg)^2 \right\}
\end{eqnarray*}
where $\Lambda_{n,0}=\langle 1 - 2 u \rangle_n$ from \eqref{firstcoeff}.
\section{Laguerre-Gamma density approximation}
Consider the Gamma pdf with scale parameter $\alpha+1>0$ and rate parameter $\beta>0$,
\begin{equation}
f_{\alpha,\beta}(t)
= \beta (\beta t)^{\alpha} \frac{e^{-\beta t}}{\Gamma(\alpha+1)},
\qquad t>0,
\label{gammapdf}
\end{equation}
and let $g(t)$ denote the FPT pdf. If $\beta < 2/\mathbb{E}[T]$ and
$
g(t) = o(t^{\delta})$ as $t \to 0,
$
for some $\delta > \alpha/2 + 1$, then $g(t)$ admits a Fourier--Laguerre expansion in terms of the complete orthonormal system of generalized Laguerre polynomials
\begin{equation}
L_k^{(\alpha)}(t)
=
\frac{\Gamma(\alpha+1+k)}{k!}
\sum_{j=0}^{k}
\binom{k}{j}
\frac{(-t)^j}{\Gamma(\alpha+j+1)},
\qquad k \ge 1,
\end{equation}
with $L_0^{(\alpha)}(t)=1$. A truncated expansion of order $n$ yields the approximant
\begin{equation}
\hat{g}_n(t)
=
\beta (\beta t)^{\alpha} e^{-\beta t} \bigg( \frac{1}{\Gamma(\alpha+1)} + \sum_{k=1}^n \mathfrak{B}_k^{(\alpha)}
L_k^{(\alpha)}(\beta t) \bigg),
\quad \hbox{
where} \,\,
\mathfrak{B}_k^{(\alpha)}
=
\sum_{j=0}^{k}
\binom{k}{j}
\frac{(-\beta)^j \, \mathbb{E}[T^j]}
{\Gamma(\alpha+j+1)}.
\label{explaguerre} 
\end{equation}
Note that the approximation \eqref{explaguerre} applies also when the FPT moments are not available in closed-form. Given a sample of FPT observations, the required moments can be estimated and substituted into \eqref{explaguerre}, yielding an orthogonal series density estimator with respect to the Gamma reference measure. Both this estimator and the associated approximation are discussed in~\cite{DOM2024}, where also convergence properties, truncation criteria, and numerical refinements
of \eqref{explaguerre} are analyzed in detail. Here, we restrict attention to the aspects relevant to the present setting.

By the orthogonality of generalized Laguerre polynomials, $\hat g_n(t)$ is normalized and this property is used
to determine the truncation order $n$ in \eqref{explaguerre}.
The parameters $\alpha$ and $\beta$ in \eqref{gammapdf} 
are chosen by moment matching,
\begin{equation}
\alpha = \frac{(\mathbb{E}[T])^2}{\mathrm{Var}(T)} - 1,
\qquad
\beta = \frac{\mathbb{E}[T]}{\mathrm{Var}(T)}.
\label{mm}
\end{equation}
With this choice, the first two coefficients
of the Fourier-Laguerre expansion vanish and
$\alpha + 1 = c_v^{-2},$ where
$c_v = \sqrt{\mathrm{Var}(T)}/\mathbb{E}[T]$ is the coefficient of
variation. Hence large dispersion ($c_v>1$) implies $\alpha<0$,
for which the Gamma pdf is singular at the origin and strongly concentrated near zero, potentially leading to numerical instability, especially for heavy-tailed FPT distributions. This issue should therefore be carefully assessed
before applying the approximation.
Alternative choices of $\alpha$ and $\beta$ were examined,
but no systematic improvement was observed.  

The use of the Laguerre-Gamma expansion can be further supported by  the FPT cumulants. For a Gamma rv $X$,
we have 
$c_{k+1}(X)/c_k(X) = k/\beta,$  for $k \ge 1.$ Thus, when the ratios between the FPT cumulants follow a similar law, this is indicative of consistency with a Gamma reference and supports the  Laguerre-Gamma approximation \eqref{explaguerre}.

 Since $\hat g_n(t)$ is expressed as the product of $f_{\alpha,\beta}(t)$ and a polynomial $p_n$ of degree $n$, its sign is entirely determined by $p_n$. Therefore positivity of $\hat g_n(t)$ is not guaranteed. Sufficient conditions on the constant and leading coefficients of $p_n$ are included in the stopping rules to ensure non-negativity near the origin and in the tails, \cite{DOM2024}. Moreover, when necessary, a local correction is applied to restore positivity, consistent with the theoretical power-law behavior near zero and exponential decay in the tail. In the numerical experiments considered here, such corrections have a negligible impact.
\subsection{Numerical results}

In this section, we present numerical results for the approximation of the FPT pdf $g(t)$ of the logistic model via $\hat{g}_n(t)$ in~\eqref{explaguerre}. Since the shape of the FPT pdf is unknown, the validity of the proposed approximation is assessed by comparison with alternative estimates obtained through different techniques.

As noted in the introduction, the \texttt{fptdApprox} package for one-dimensional diffusion processes cannot be applied in this case. Therefore, Monte Carlo simulation is adopted as the reference method to evaluate the accuracy of the proposed approximation.
This method consists in simulating sample paths of $X(t)$ and extracting the corresponding
FPT instances to the threshold.
The simulation is performed using a Lie-Trotter splitting of the generators associated with the drift and diffusion operators, which preserves the positivity of the process (see, e.g., \cite{BensoussanGlowinskiRascanu1992}). The scheme separates the deterministic drift and stochastic diffusion components in
${\rm d} X(t) = a[X(t)] \,{\rm d}t + \sigma X(t) \, {\rm d} W(t).$
Over a time step $\Delta t$, the drift component $\dot{x} = a(x)$ is integrated exactly, yielding the intermediate value
$
X_n^* = \Phi_{\Delta t}(X_n),
$
where $\Phi_{\Delta t}$ denotes the flow associated with the deterministic equation.
The diffusion component
${\rm d} X(t) = \sigma X(t) \, {\rm d} W(t)$
corresponds to a geometric Brownian motion without drift and admits an explicit solution. The numerical update is then obtained by composing the two flows:
$$
X_{n+1}
=
X_n^*
\exp\!\left(
-\frac{1}{2}\sigma^2 \Delta t
+
\sigma \sqrt{\Delta t}\, Z_n
\right),
\qquad Z_n \sim \mathcal{N}(0,1).
$$
Accurate estimation requires a large number of trajectories and a
sufficiently small time step, which substantially increases the
computational cost.
Furthermore, since not all simulated paths reach the threshold within a
finite simulation horizon, the FPT pdf may be underestimated.
The computational burden becomes particularly pronounced for large
coefficients of variation, as trajectories must be simulated over
longer time intervals. 

For all figures, the dashed curve represents the kernel density estimate of the FPT pdf obtained from Monte Carlo simulations. The solid curve corresponds to the Laguerre–Gamma approximation \eqref{explaguerre} based on theoretical moments, while the dotted curve is obtained from \eqref{explaguerre} replacing theoretical moments with empirical moments computed from the same sample
obtained with the Monte Carlo method. The {\tt R} procedures are available from the authors upon request.

\begin{figure}[H]
\centering

\begin{subfigure}{0.48\linewidth}
    \centering
    \includegraphics[width=\linewidth]{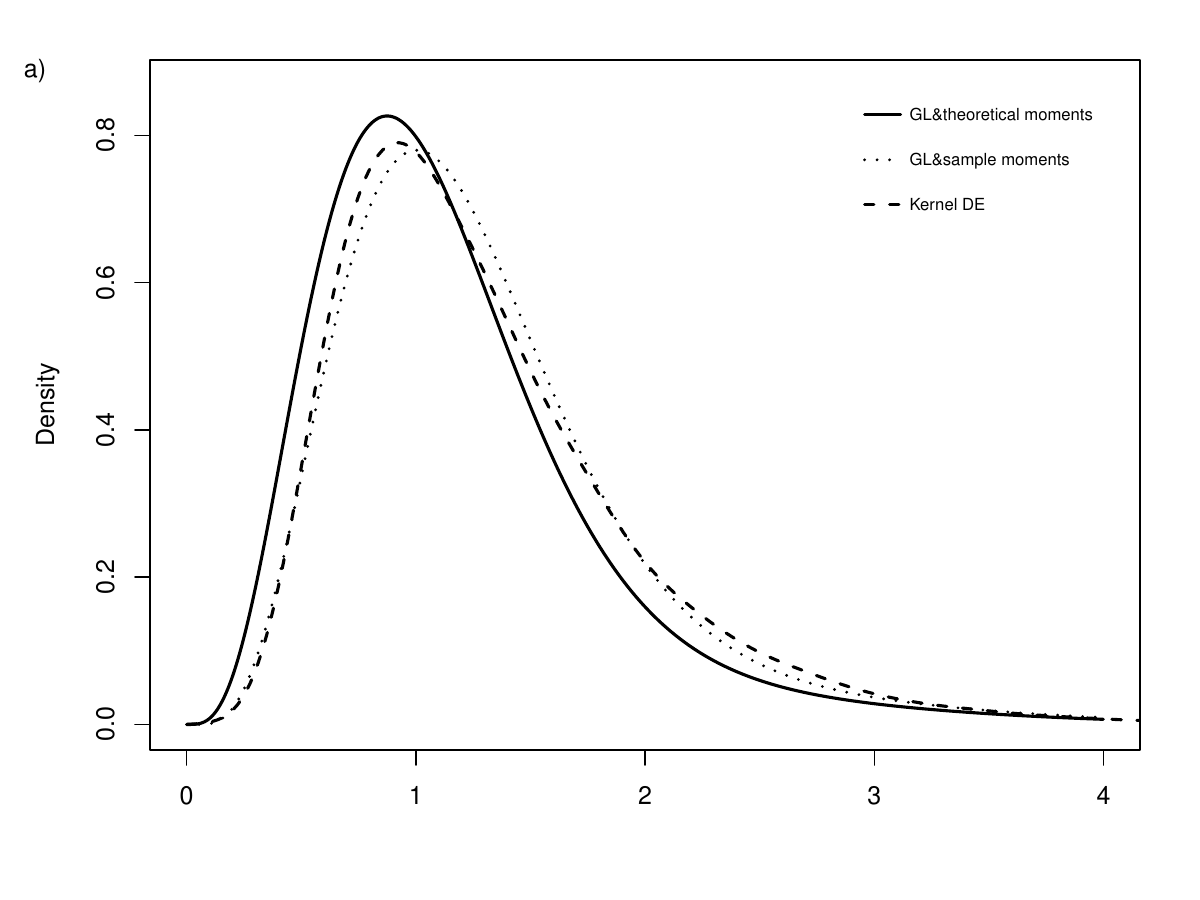}
\end{subfigure}
\hfill
\begin{subfigure}{0.48\linewidth}
    \centering
    \includegraphics[width=\linewidth]{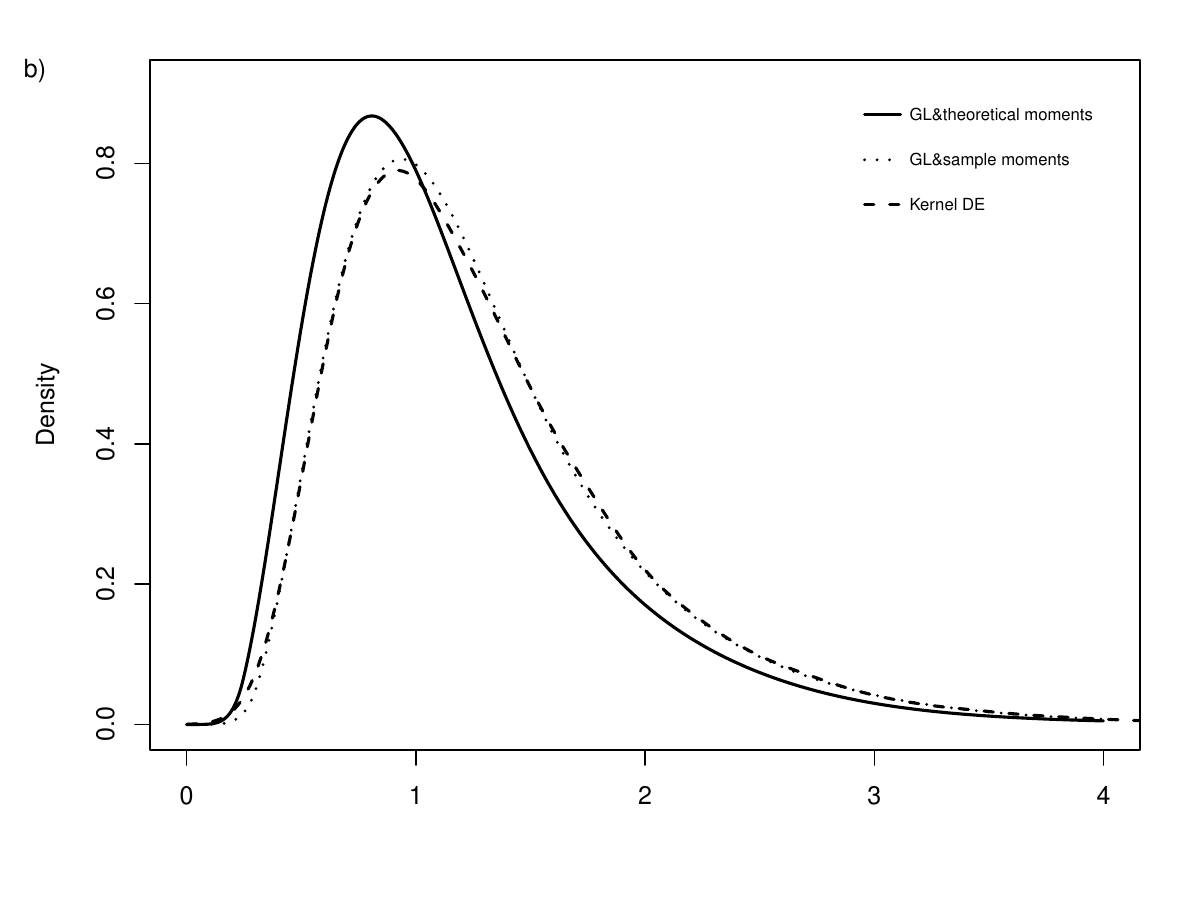}
\end{subfigure}

\caption{Plots of approximated FPT pdfs for $x_0 = 100, \, U = 150, \, \alpha = 2.50, \, \beta = 2.98, \, c_{v} = 0.53$ with $n=4$ in a) and $n=8$ in b), same line styles as described above.}
\label{Fig1}
\end{figure}

\begin{figure}[H]
\centering

\begin{subfigure}{0.48\linewidth}
    \centering
    \includegraphics[width=\linewidth]{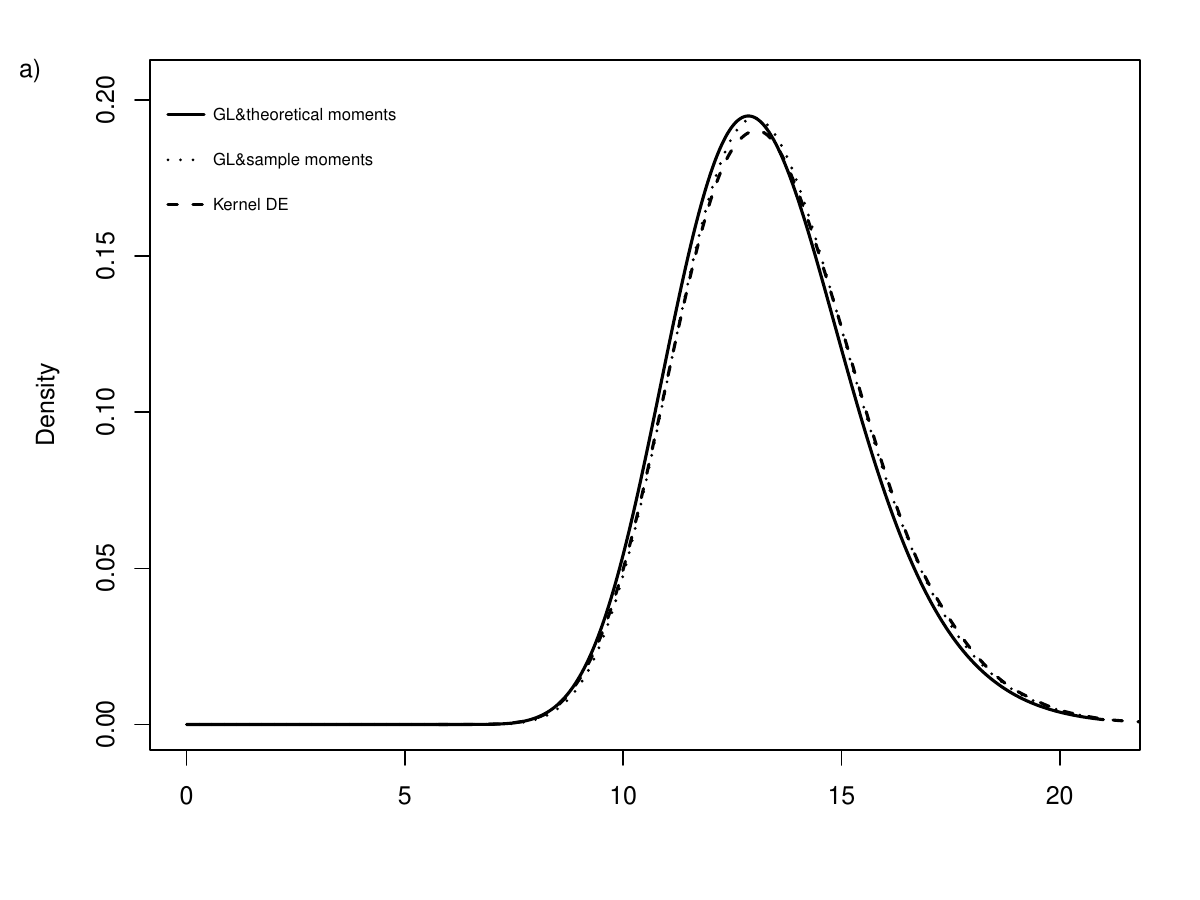}
\end{subfigure}
\hfill
\begin{subfigure}{0.48\linewidth}
    \centering
    \includegraphics[width=\linewidth]{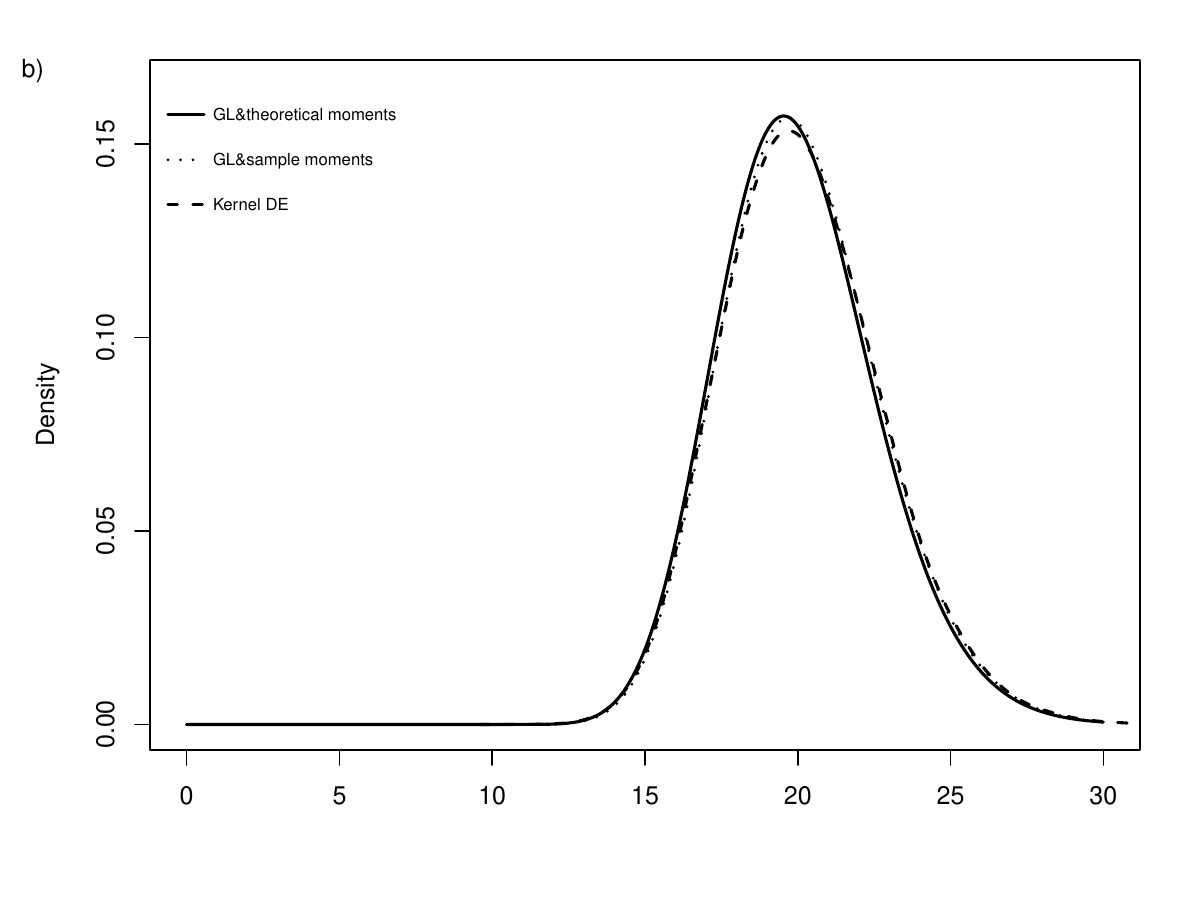}
\end{subfigure}

    \caption{Plots of approximated FPT pdfs for $x_0 = 100, \, U = 10^4, \, \alpha = 38.72, \, \beta = 2.98, \, c_v = 0.16, n=4$ in a) and $x_0 = 100, U = 10^5, \alpha = 58.56, \beta = 2.97, \, c_v = 0.13$, $n=4$ in b), same line styles as described above.} \label{Fig2}
\end{figure}

For moderate dispersion ($c_v<1$), as in Figs~\ref{Fig1} and~\ref{Fig2}, the Laguerre-Gamma approximation exhibits excellent
agreement with the Monte Carlo benchmark.
In this regime, $\alpha>0$ and the reference Gamma pdf is smooth
at the origin.
Both the theoretical-moment and sample-moment implementations
essentially overlap with the Monte Carlo curve, indicating numerical
stability and robustness with respect to moment estimation.
Moreover, the cumulant ratios remain approximately constant (see Table \ref{Tab1}),
consistent with the Gamma cumulant structure.

\begin{figure}[H]
\centering

\begin{subfigure}{0.45\textwidth}
    \centering
    \includegraphics[width=\linewidth]{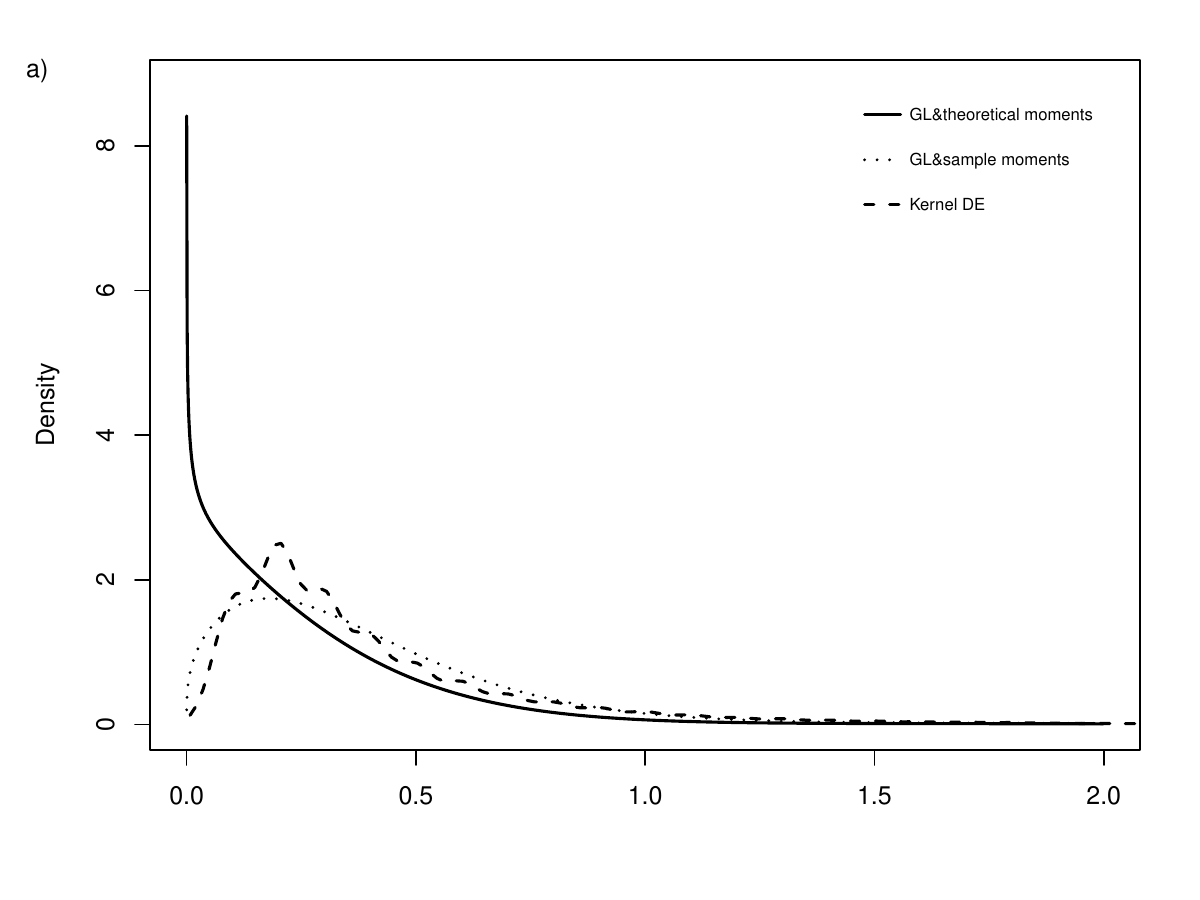}
\end{subfigure}
\hfill
\begin{subfigure}{0.45\textwidth}
    \centering
    \includegraphics[width=\linewidth]{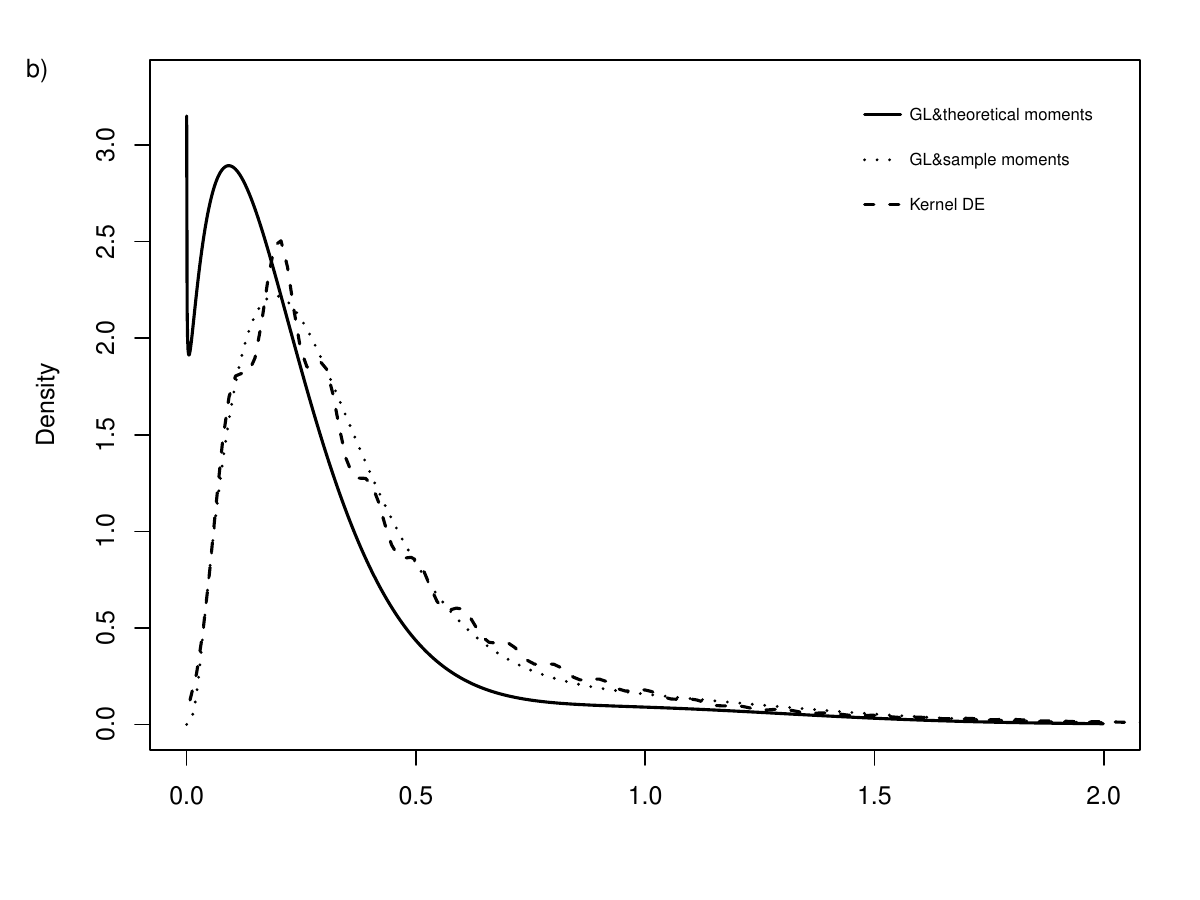}
\end{subfigure}

\vspace{0.2cm}

\begin{subfigure}{0.45\textwidth}
    \centering
    \includegraphics[width=\linewidth]{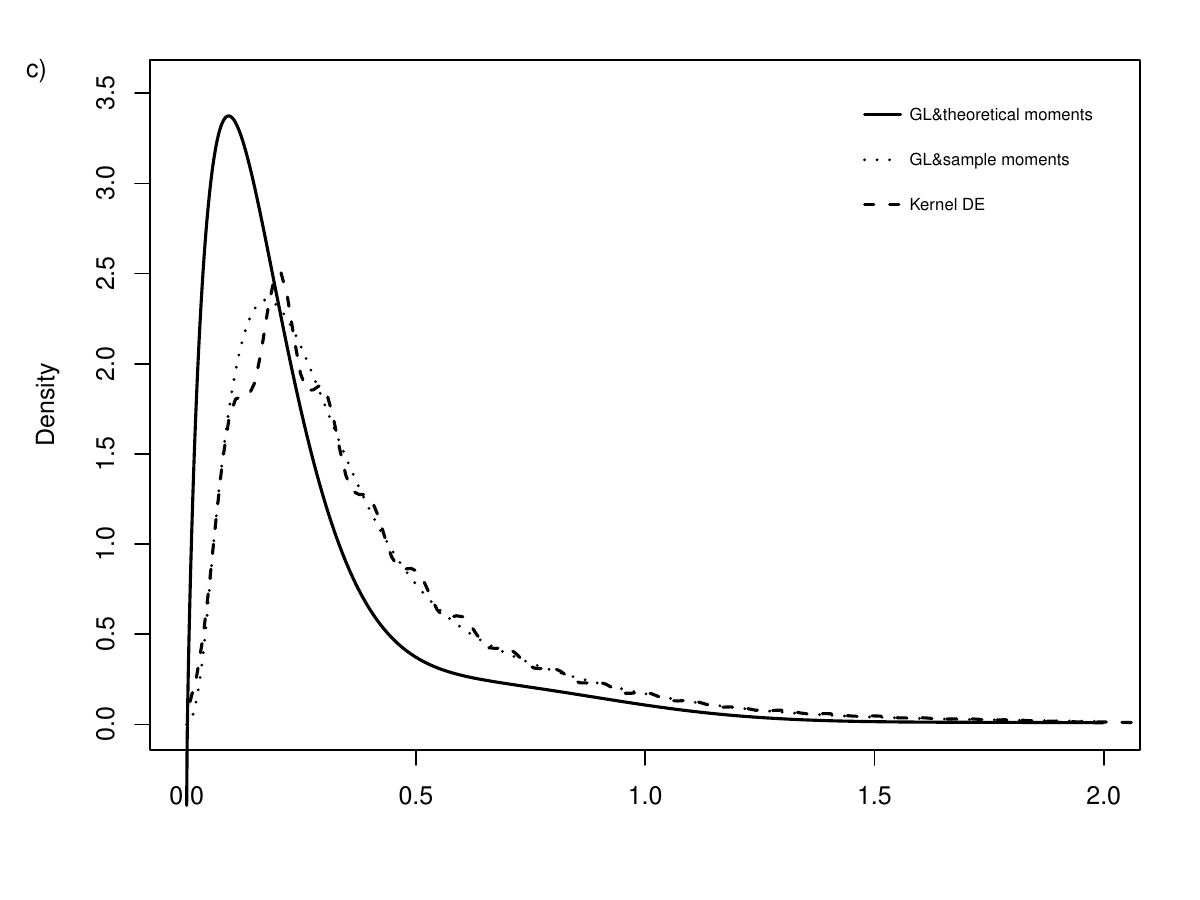}
\end{subfigure}
\hfill
\begin{subfigure}{0.45\textwidth}
    \centering
    \includegraphics[width=\linewidth]{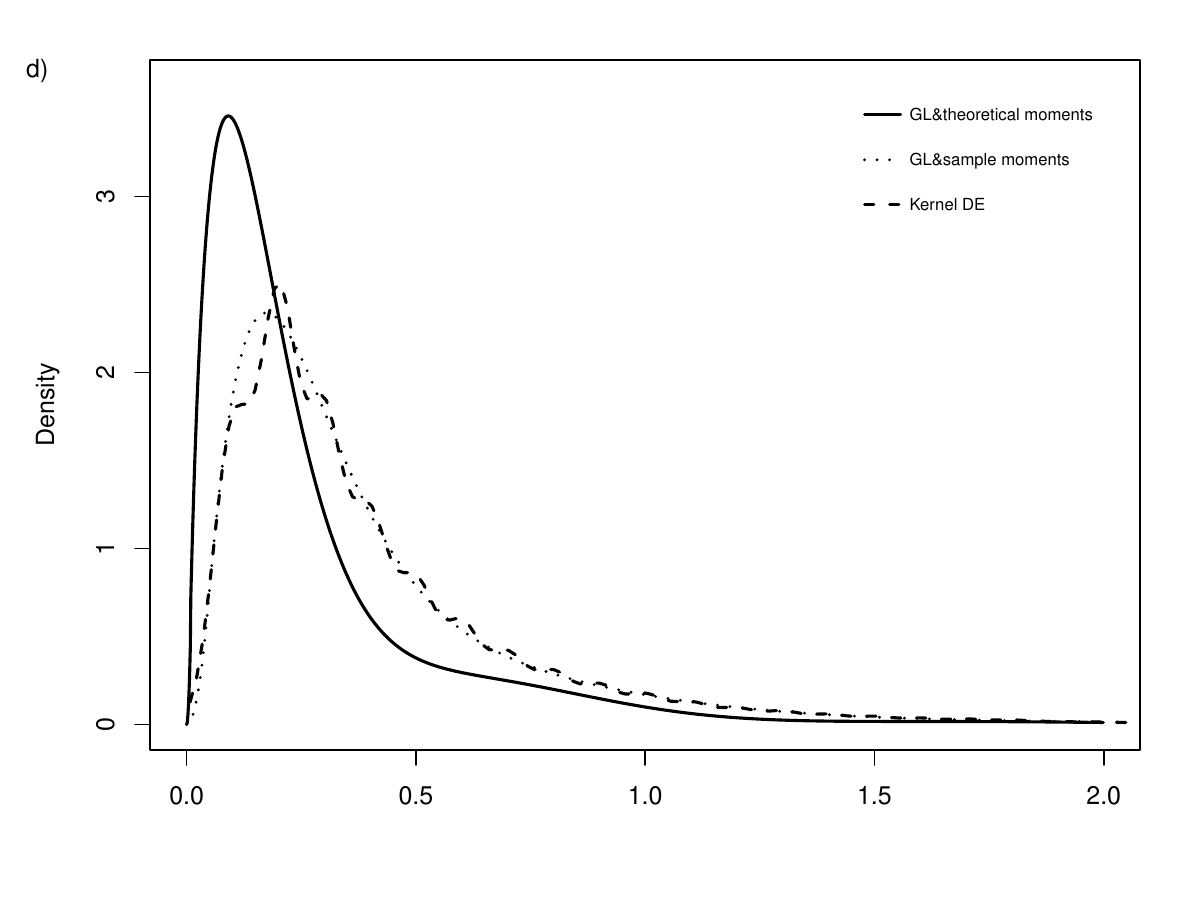}
\end{subfigure}

\caption{Plots of approximated FPT pdfs for $x_0 = 100, \, U = 110, \, \alpha = -0.18, \, \beta = 2.98, \, c_v = 1.10$, with $n=4$ in a), $n=8$ in b), $n=12$ in c) and $n=16$ in d), same line styles as described above.}
\label{Fig3}
\end{figure}

When the coefficient of variation exceeds unity, as in Fig. \ref{Fig3}, the moment-matching procedure yields $\alpha<0$. In this case, the reference Gamma pdf exhibits a singular behavior at the origin, and the truncated expansion becomes more sensitive to the choice of the truncation order $n$. Increasing the number of moments might improve the fit, as in Fig. \ref{Fig3}.
The Monte Carlo estimate appears to display a multimodal shape. However, since the FPT pdf is known to be unimodal, the observed oscillations are attributable to sampling variability and discretization effects.

\subsection{Applications}
\subparagraph{Fisheries management.}
Figs.~\ref{Fig4}--\ref{Fig6} show the Laguerre–Gamma approximation of the FPT pdf for the stochastic logistic model with constant-effort harvesting, arising in fisheries management under random environments \citep{BritesBraumann2022}. Lower and upper thresholds represent, respectively, biological warning levels and recovery targets. The results refer to the upcrossing problem; downcrossing cases are omitted as they show analogous qualitative results. The large values of $x_0$ and $U$ reflect the scale of the application, but the approximation depends mainly on the coefficient of variation and is therefore unaffected by the magnitude of the state variable. Figs.~\ref{Fig4}--\ref{Fig6}  show  again the inverse relationship between dispersion and approximation accuracy: smaller values of $c_v$ yield improved agreement with the Monte Carlo density, while larger $c_v$ increase truncation effects and numerical sensitivity.

\begin{figure}[H]
\centering

\begin{subfigure}{0.48\linewidth}
    \centering
    \includegraphics[width=\linewidth]{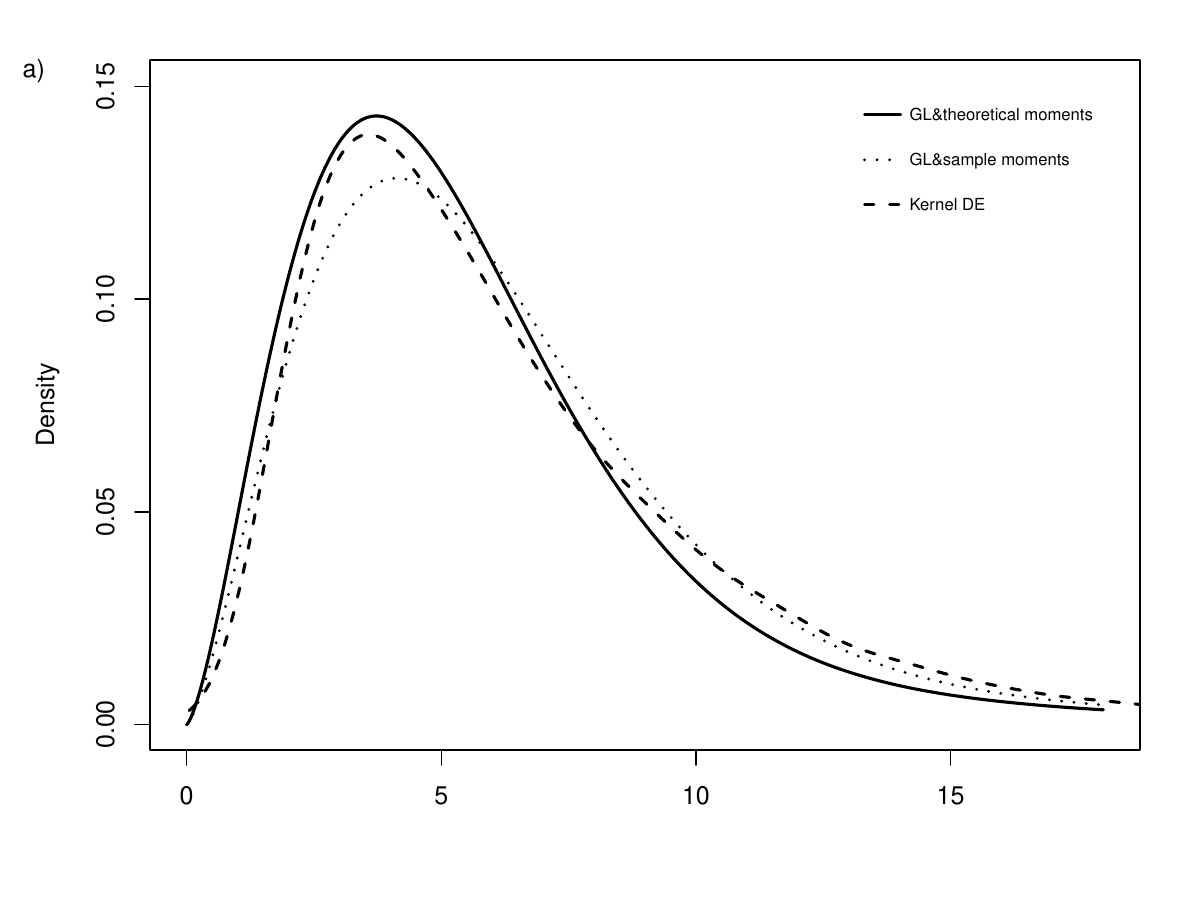}
\end{subfigure}
\hfill
\begin{subfigure}{0.48\linewidth}
    \centering
    \includegraphics[width=\linewidth]{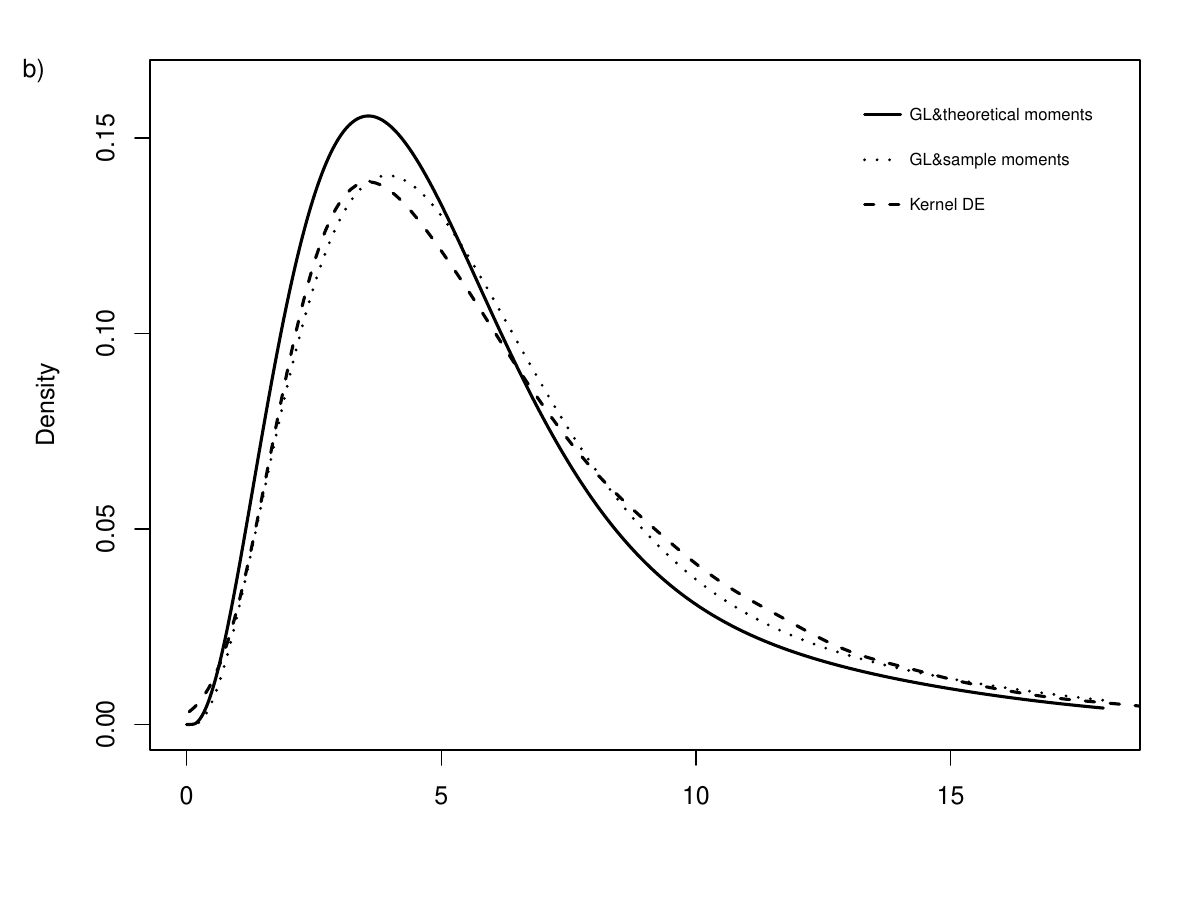}
\end{subfigure}
\caption{Plots of approximated FPT pdfs for $x_0 = 2.01 \times 10^7, U = 3.91 \times 10^7, \alpha = 1.36, \beta = 0.41, c_v = 0.65$ for $n=4$ in a), $n=8$ in b), same line styles as described above.}.
\label{Fig4}
\end{figure}

\begin{figure}[H]
\centering

\begin{subfigure}{0.48\linewidth}
    \centering
    \includegraphics[width=\linewidth]{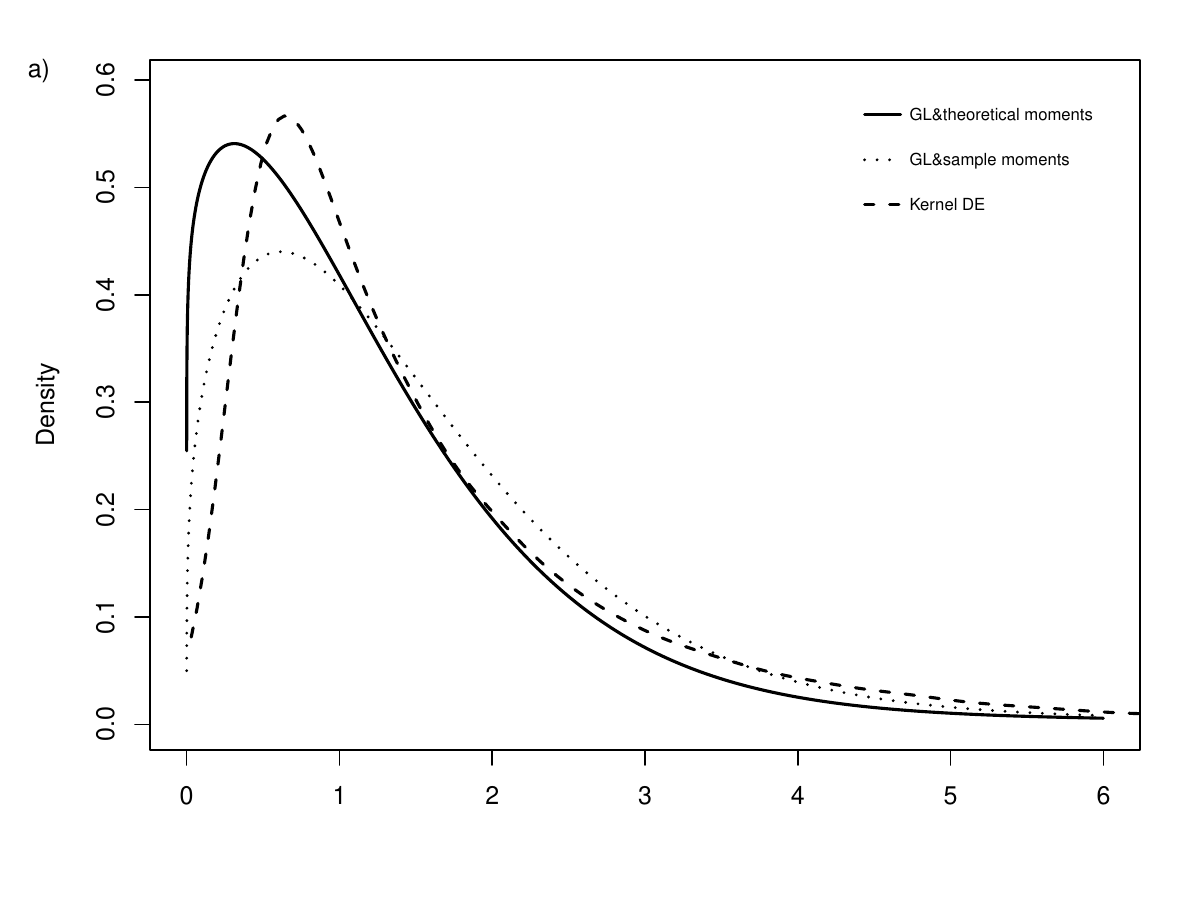}
\end{subfigure}
\hfill
\begin{subfigure}{0.48\linewidth}
    \centering
    \includegraphics[width=\linewidth]{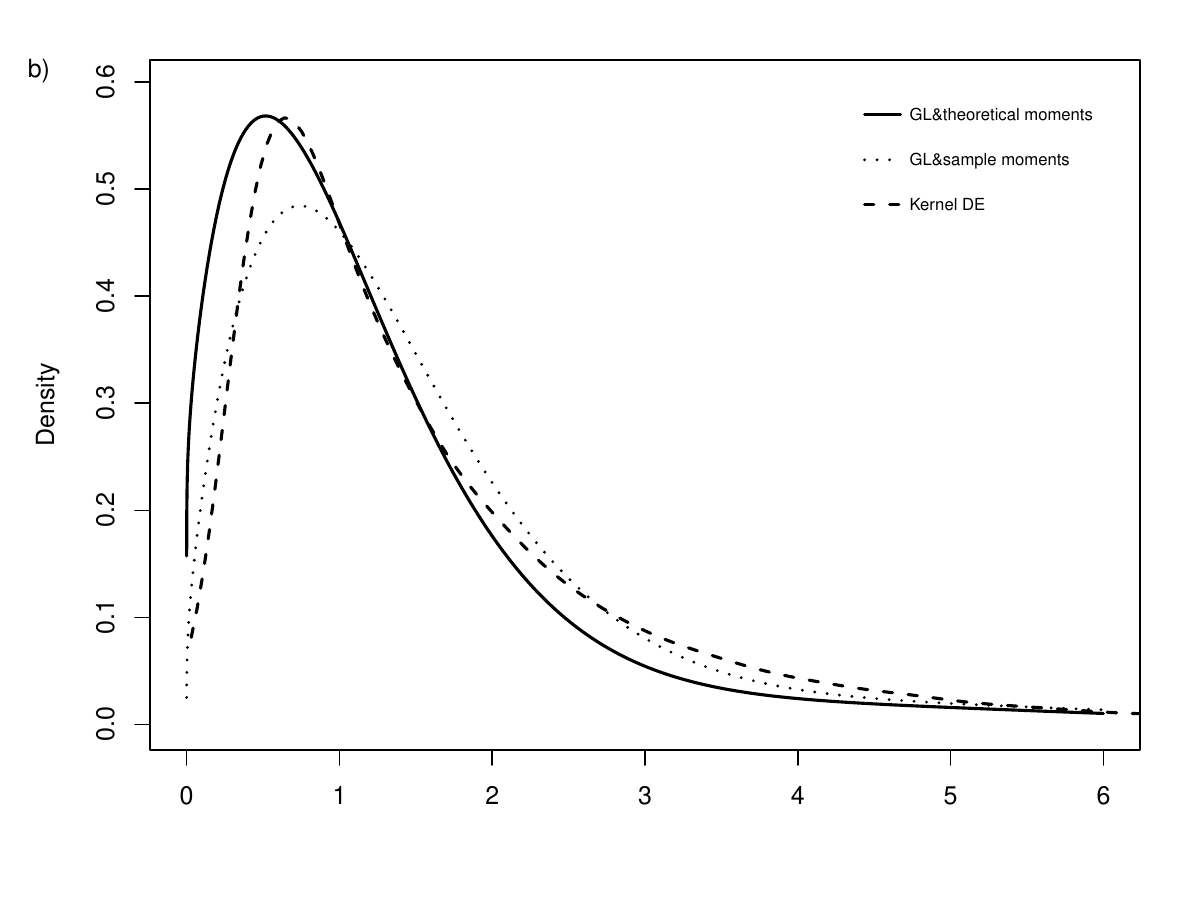}
\end{subfigure}
\caption{Plots of approximated FPT pdfs for $x_0 = 2.01 \times 10^7, \, U = 2.51 \times 10^7, \alpha = 0.10, \beta = 0.82, c_v = 0.95$ for $n=4$ in a), $n=6$ in b), same line styles as described above.}.
\label{Fig5}
\end{figure}

\begin{figure}[H]
\centering

\begin{subfigure}{0.45\textwidth}
    \centering
    \includegraphics[width=\linewidth]{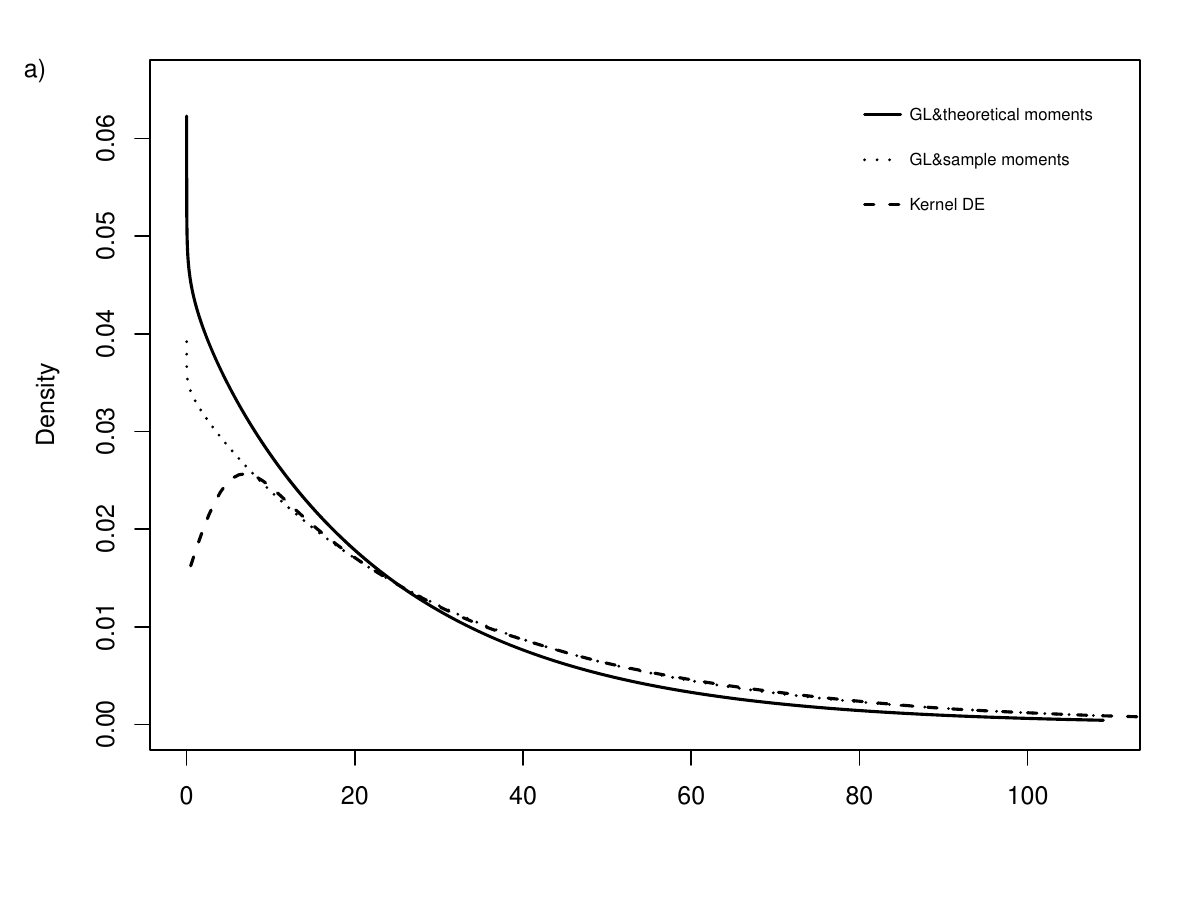}
\end{subfigure}
\hfill
\begin{subfigure}{0.45\textwidth}
    \centering
    \includegraphics[width=\linewidth]{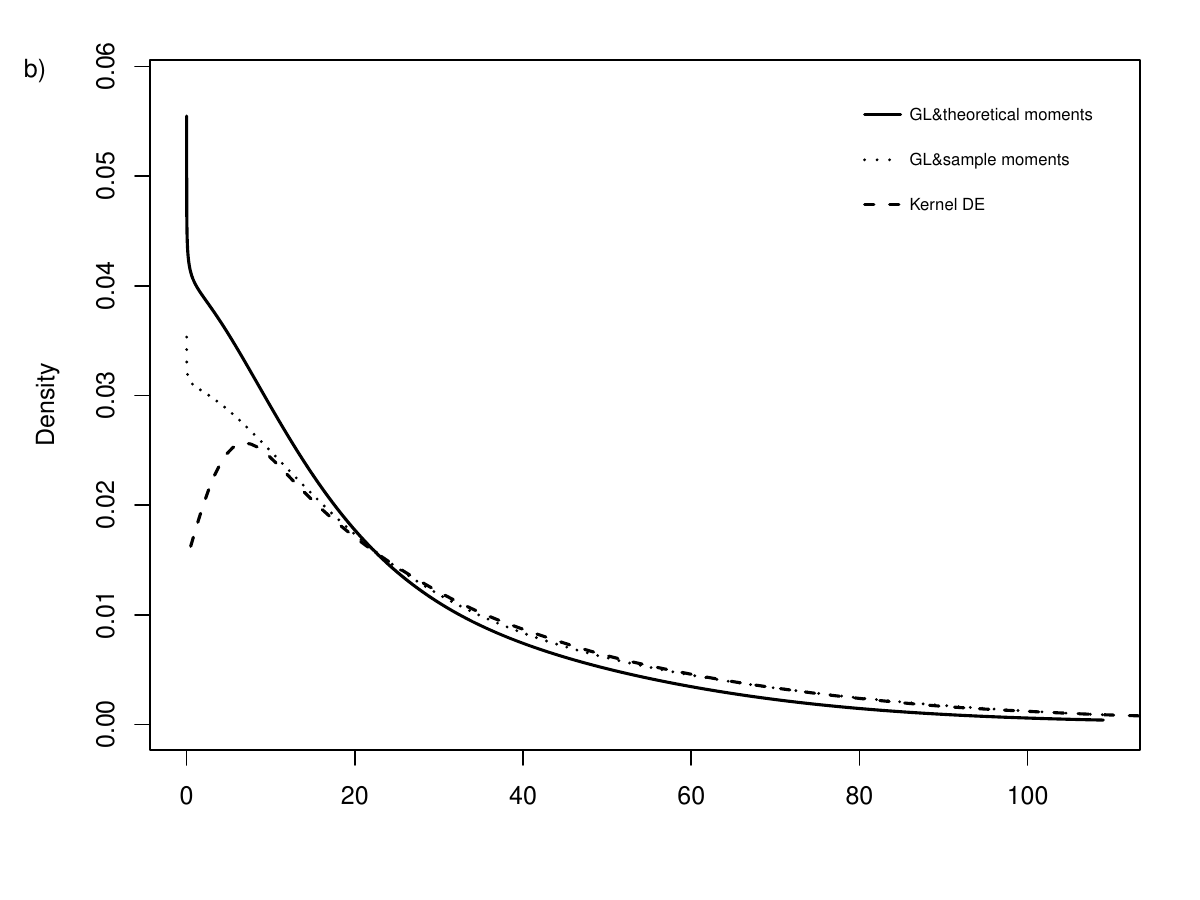}
\end{subfigure}

\vspace{0.2cm}

\begin{subfigure}{0.45\textwidth}
    \centering
    \includegraphics[width=\linewidth]{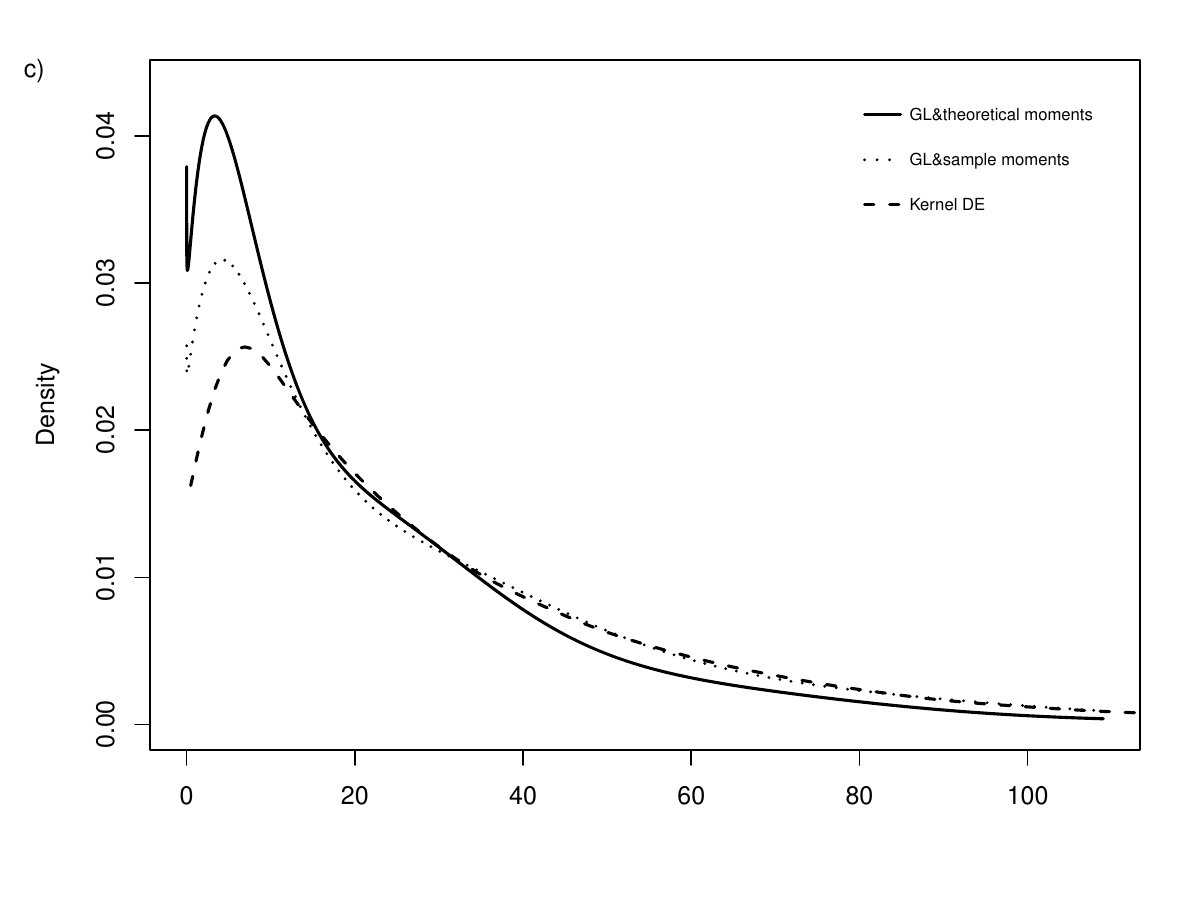}
\end{subfigure}
\hfill
\begin{subfigure}{0.45\textwidth}
    \centering
    \includegraphics[width=\linewidth]{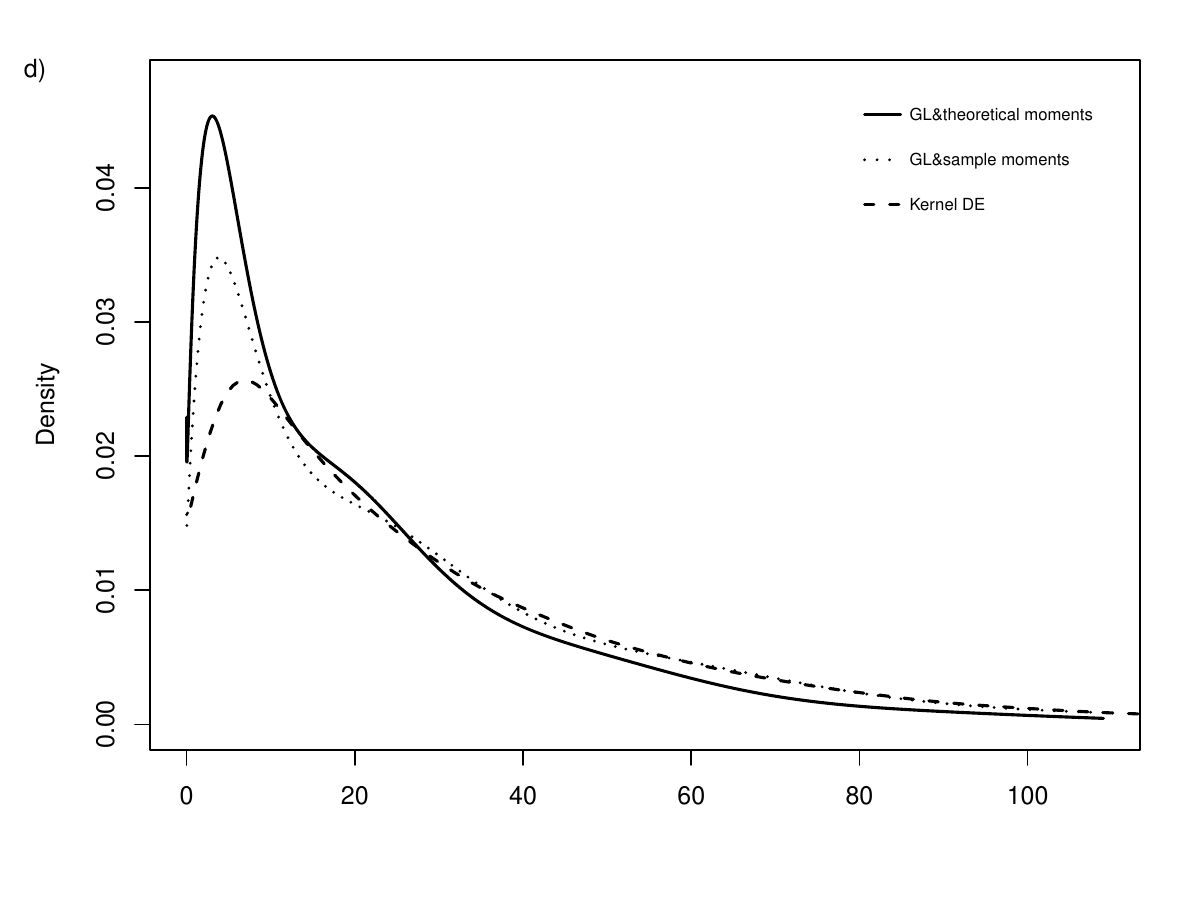}
\end{subfigure}
\caption{Plots of approximated FPT pdfs for $x_0 = 4  \times 10^7, U = 6 \times 10^7, \alpha = -0.04, \beta = 0.04, c_v = 1.02$ for $n=4$ in a), $n=12$ in b), $n=24$ in c) and $n=36$ in d), same line styles as described above.}.
\label{Fig6}
\end{figure}

Figure~\ref{Fig6} corresponds to a regime that may be numerically more demanding. The large higher-order cumulants (of order $10^6$, see Table~\ref{Tab1}) indicate extreme dispersion, leading to a sharply peaked and highly skewed density. This degrades the approximation, as rapidly growing moments increase the Laguerre coefficients and reduce numerical stability.
\begin{table}[H]
\centering
\caption{Summary statistics and cumulant structure for the six scenarios shown in Figures 1--6.}
\begin{tabular}{c c c c c c c c c}
\hline
Fig. & Mean & Var & Sample Mean & Sample Var 
& $\kappa_4$ 
& $\kappa_1/\kappa_2$ 
& $2\kappa_2/\kappa_3$ 
& $3\kappa_3/\kappa_4$ \\
\hline

1  & 1.18  & 0.39  & 1.32  & 0.42  & 0.67   
   & 2.98  & 1.98  & 1.79 \\

2L & 13.35 & 4.49  & 13.68 & 4.40  & 7.60  
   & 2.98  & 1.98  & 1.79 \\

2R & 20.03 & 6.73  & 20.47 & 6.59  & 11.42 
   & 2.97  & 1.98  & 1.78 \\

3  & 0.28  & 0.09  & 0.41  & 0.12  & 0.16  
   & 2.98  & 1.98  & 1.79 \\

4  & 1.33  & 1.62  & 1.62  & 2.02  & 25.91 
   & 0.82  & 0.62  & 0.60 \\

5  & 5.78  & 14.19 & 6.40  & 17.24 & 963.23
   & 0.41  & 0.31  & 0.29 \\

6  & 23.40 & 567.79 & 30.06 & 915.49 & 1990822
   & 0.04  & 0.04  & 0.04 \\

\hline
\end{tabular}
\label{Tab1}
\end{table}

\subparagraph{Maximum likelihood estimation.}

We investigate the use of approximant $\hat g_n$ \eqref{explaguerre} in place of the unknown true FPT pdf within the standard maximum likelihood framework. The log-likelihood function has been maximized numerically using the function {\tt Optim} from the base {\tt R} package {\tt stats}. Technical details can be found in \cite{Belisle1992}, which underlies the algorithm implemented in {\tt Optim}.

Given the relatively large number of parameters $(r, K, E, q, \sigma, x_0, U)$, we fix a subset of parameters and focus on the joint estimation of this smaller set. To assess the MLE performance, a sample of $N \in \{100, 500, 1000\}$ FPTs is generated using the values
\[
r = 0.71, \quad K = 80.5 \times 10^6, \quad q = 3.30 \times 10^{-6}, \quad \sigma = 0.2, \quad E = 104540, \quad x_0 = 100,\quad U = 10^4.
\]
Note that the approximation order $n$ and the parameters $\alpha$ and $\beta$ must be specified when computing the approximant $\hat g_n$. A maximum order $n_{\max}$ is fixed at the outset of the procedure, and the iterative scheme is terminated either when the stopping criteria are satisfied or when $n_{\max}$ is reached. To limit computational cost, $n_{\max}$ is set to $10$ in all cases considered.
The parameters $\alpha$ and $\beta$ are initially determined by moment matching based on sample moments, and subsequently updated at each step of the procedure by matching the theoretical moments corresponding to the estimated current parameter values \citep{Martini2024} .

Table~\ref{Tab2} shows the results when  $\sigma$ (the environmental variability) is estimated jointly with $r$ (the intrinsic growth rate), $x_0$ (the size population) and $U$ (the upcrossing threshold).  
While the method provides stable estimates for the pairs $(\sigma, x_0)$, $(\sigma, U)$, and $(\sigma, r)$, the estimation becomes challenging for other combinations, with convergence issues.
\begin{table}[H]
\centering
\resizebox{\textwidth}{!}{
\begin{tabular}{c c ccc ccc ccc}
\toprule
 &  & \multicolumn{3}{c}{$(\sigma,x_0)$} 
   & \multicolumn{3}{c}{$(\sigma,U)$} 
   & \multicolumn{3}{c}{$(\sigma,r)$} \\
\cmidrule(lr){3-5} \cmidrule(lr){6-8} \cmidrule(lr){9-11}
$N$ & Param. & Bias & MSE & Err.(\%) 
            & Bias & MSE & Err.(\%) 
            & Bias & MSE & Err.(\%) \\
\midrule

\multirow{2}{*}{100}
& $\sigma$ 
& $-1.99 \times 10^{-3}$ & $1.76 \times 10^{-4}$ & $5.31$
& $6.62 \times 10^{-3}$  & $2.66 \times 10^{-3}$ & $9.64$
& $-1.85 \times 10^{-3}$ & $1.83 \times 10^{-4}$ & $5.31$ \\

& second 
& $-5.04$ & $8.57 \times 10^{1}$ & $7.22$
& $2.19 \times 10^{2}$ & $3.87 \times 10^{6}$ & $10.3$
& $-1.38 \times 10^{-3}$ & $5.00 \times 10^{-5}$ & $0.80$ \\

\midrule

\multirow{2}{*}{500}
& $\sigma$ 
& $-5.53 \times 10^{-4}$ & $3.47 \times 10^{-5}$ & $2.37$
& $3.96 \times 10^{-5}$  & $4.60 \times 10^{-5}$ & $2.43$
& $-2.24 \times 10^{-3}$ & $4.41 \times 10^{-5}$ & $2.63$ \\

& second 
& $-4.62$ & $3.13 \times 10^{1}$ & $4.86$
& $4.61 \times 10^{2}$ & $4.12 \times 10^{5}$ & $5.21$
& $-3.71 \times 10^{-3}$ & $2.22 \times 10^{-5}$ & $0.57$ \\

\midrule

\multirow{2}{*}{1000}
& $\sigma$ 
& $-2.52 \times 10^{-4}$ & $1.86 \times 10^{-5}$ & $1.71$
& $7.75 \times 10^{-4}$  & $1.65 \times 10^{-4}$ & $2.23$
& $-2.23 \times 10^{-3}$ & $2.79 \times 10^{-5}$ & $2.09$ \\

& second 
& $-4.61$ & $2.87 \times 10^{1}$ & $4.64$
& $4.38 \times 10^{2}$ & $6.59 \times 10^{5}$ & $5.47$
& $-3.82 \times 10^{-3}$ & $1.87 \times 10^{-5}$ & $0.55$ \\

\bottomrule
\end{tabular}
}
\caption{Maximum likelihood estimations for different sample sizes $N$ and parameter pairs.}
\label{Tab2}
\end{table}

We then extend the analysis to the estimation of three parameters, $(\sigma, r, x_0)$, which is the only triplet exhibiting stable convergence properties, in order to investigate the impact of increased parameter dimensionality on the accuracy and stability of the MLE.

\begin{table}[H]
\centering
\resizebox{0.6\textwidth}{!}{
\begin{tabular}{cccccc}
\toprule
$N$ & Parameter & Bias & MSE & Error (\%) \\
\midrule
\multirow{3}{*}{100} 
& $\sigma$ & $9.00 \times 10^{-2}$ & $1.15 \times 10^{-2}$ & $45.5$ \\
& $r$      & $-1.84\times 10^{-1}$ & $4.69 \times 10^{-2}$ & $25.9$ \\
& $x_0$    & $-7.59\times 10^{1}$  & $6.47 \times 10^{3}$ & $75.9$ \\
\midrule
\multirow{3}{*}{500} 
& $\sigma$ & $-2.14 \times 10^{-2}$ & $5.06 \times 10^{-4}$ & $10.7$ \\
& $r$      & $4.21\times 10^{-2}$  & $1.78 \times 10^{-3}$ & $5.93$ \\
& $x_0$    & $-4.26 \times 10^{1}$ & $1.82 \times 10^{3}$ & $42.6$ \\
\midrule
\multirow{3}{*}{1000} 
& $\sigma$ & $-2.34 \times 10^{-2}$ & $5.87 \times 10^{-4}$ & $11.7$ \\
& $r$      & $-4.39\times 10^{-2}$ & $2.00 \times 10^{-3}$ & $6.18$ \\
& $x_0$    & $6.40 \times 10^{1}$  & $4.56 \times 10^{3}$ & $64.0$ \\
\bottomrule
\end{tabular}
}
\caption{Maximum likelihood estimations for $\sigma$, $r$ and $x_0$ and different sample sizes $N$}
\end{table}
\section{Conclusions}

We have developed a moment-based approach to recover FPT pdfs for stochastic logistic models with constant harvesting. A key contribution is the derivation of closed-form expressions for the FPT moments and cumulants, obtained through a suitable power series expansion of the Laplace transform of the FPT pdf. 
Building on these results, we construct an explicit representation of the FPT pdf in terms of a Laguerre–Gamma expansion, using a Gamma pdf as reference to match the positive support and skewed shape of the FPT pdf. 

From a numerical perspective, the proposed representation can be effectively employed to approximate the FPT pdf by selecting a truncation order that balances accuracy and numerical stability. The method performs particularly well in regimes of moderate dispersion, where the cumulant structure is close to that of a Gamma distribution, while its performance deteriorates when $c_v>1$, mainly because higher-order moments grow rapidly, leading to numerical instability in the evaluation of the truncated series.

The application to fisheries management models shows that the method remains effective even for large-scale population levels, as its performance primarily depends on dispersion. Moreover, 
the parameter estimations  further indicate that the proposed approximation can be used with the maximum likelihood method. In particular, this approach appears to be a promising technique worthy of further investigation, although its numerical stability and robustness require additional analysis, especially with respect to parameter identifiability and conditioning of the likelihood surface.

Future work may focus on the development of refined criteria for selecting the truncation order, for instance by exploiting the moment-matching property between the true and approximated densities, as well as on adaptive truncation strategies. Further extensions to more general stochastic models and boundary conditions are also of interest. 

\appendix
\renewcommand{\theequation}{\thesection.\arabic{equation}}

\section{Appendix}
\setcounter{equation}{0}

\renewcommand{\theHequation}{\thesection.\arabic{equation}}

\subparagraph{Algebra of formal power series.}  
If $A(t)=\sum_{n\ge 0} a_nt^n/n!$ and $B(t)=\sum_{n\ge 0} b_n t^n/n!,$ the composition $A(B(t))=\sum_{n\ge 0} c_n t^n/n!,$ with $b_0=0,$ has coefficients
$c_n=\sum_{k=0}^n a_k\,B_{n,k}\!\bigl(b_1,b_2,\dots \bigr.$ $\bigl. b_{n-k+1}\bigr), (n\ge 0),$ (with $B_{0,0}=1$ and $B_{n,0}=0$ for $n\ge 1$)
where
\begin{equation} 
B_{n,k}(b_1,\dots,b_{n-k+1})
=  n! \sum
\prod_{m=1}^{n-k+1} \frac{b_m^{\,j_m}}{(m!)^{j_m}\,j_m!},
\label{polinc}
\end{equation}
with the sum running over all  integers $(j_1,\dots,j_{n-k+1})$ such that
$j_1+j_2+\cdots+j_{n-k+1}=k,$ and $j_1+2j_2+\cdots+(n-k+1)j_{n-k+1}=n.$ 
The polynomials $\{B_{n,k}\}$ are the (exponential) incomplete Bell polynomials. 
When $A(t)=\exp(t)$, the coefficients $\{c_n\}$ reduce to the
complete Bell polynomials, namely
\begin{equation}
B_n(b_1,\dots,b_n)=\sum_{k=1}^n B_{n,k}(b_1,\dots,x_{n-k+1})\quad(n\ge 1),\qquad B_0=1.
\label{completeBell}
\end{equation}
If $a_0 \ne 0$, we have
\begin{equation} 
\log A(t) = \log a_0 + \sum_{n\ge 1} L_n(a_0;a_1,\dots,a_n)\,
\frac{t^n}{n!},
\label{logexpansion}
\end{equation}
where $\{L_n\}$ are the logarithmic polynomials given by
\begin{equation}
L_n(a_0;a_1,\dots,a_n)
= \sum_{k=1}^n (-1)^{k-1}(k-1)!\, B_{n,k}\!\Bigl(
\frac{a_1}{a_0},\frac{a_2}{a_0},\dots,\frac{a_{n-k+1}}{a_0}\Bigr),\qquad n\ge 1.
\label{logpoly}
\end{equation}
The reciprocal $\frac{1}{A(t)}=\sum_{n\ge 0} r_n t^n/n!$ 
is determined by the coefficient recursion
\begin{equation}
r_0=\frac{1}{a_0},\qquad r_n=-\frac{1}{a_0}\sum_{j=1}^n \binom{n}{j}a_j\,r_{n-j}\quad(n\ge 1).
\label{reciproco}
\end{equation}
The coefficients $\{r_n\}$
also admits the closed-form in terms of Bell polynomials:
\begin{equation}
r_n=\frac{1}{a_0}\sum_{k=0}^n (-1)_k\,
B_{n,k}\!\Bigl(\frac{a_1}{a_0},\frac{a_2}{a_0},\dots,\frac{a_{n-k+1}}{a_0}\Bigr),
\label{reciprocochiuso}
\end{equation}
where $( -1 )_k= (-1)^k k!$ and $( -1 )_0=1.$ 
More generally, if $A(t)/B(t)=\sum_{n\ge 0} q_n t^n/n!$ 
the coefficients $\{q_n\}$ satisfy the recursion
\begin{equation}
q_0=\frac{a_0}{b_0},\qquad
q_n=\frac{1}{b_0}\left(a_n-\sum_{k=1}^n \binom{n}{k}\,b_k\,q_{n-k}\right)
\quad(n\ge 1).
\label{ratio}
\end{equation}
The coefficients $\{q_n\}$
also admits the closed-form in terms of Bell polynomials:
\[
q_n=\frac{1}{b_0}\sum_{j=0}^n \binom{n}{j}\left(
\sum_{k=0}^{\,j} (-1)_k\,
B_{j,k}\!\Bigl(\frac{b_1}{b_0},\frac{b_2}{b_0},\dots,\frac{b_{j-k+1}}{b_0}\Bigr)\right)\, a_{n-j}
\qquad(n\ge 0).
\]

\section{Appendix}
\setcounter{equation}{0}
\begin{proof}[\bf Proof (Theorem 1)]
\label{proof:upper}
The power series expansion of the Kummer function in \eqref{GU}
 yields
\begin{equation} \label{ap:phiexpansion}
\Phi\left( u[1-s(\lambda)], \, 
1-2us(\lambda), \, v y \right) = 
\sum_{n \geq 0}^{\infty} \frac{\langle  u[1-s(\lambda)] \rangle_n}{\langle1-2us(\lambda) \rangle_n} \cdot \frac{(v y)^n}{n!},
\end{equation}
where $\langle x \rangle_n = x ( x+1) \cdots (x+n-1)$ denotes the \textit{rising factorial}. We now compute the coefficients 
$\tilde{\Lambda}_{n,m}$ 
in the expansion
\begin{equation}
\langle u[1 - s(\lambda)] \rangle_n =  \sum_{m\geq0} \tilde{\Lambda}_{n,m} \frac{a^m\lambda^m}{m!}.
\label{num}
\end{equation}
Let $H(t)=\langle u(1-t) \rangle_n$ and write $H(t) = \sum_{k=0}^n c_k t^k.$ Then 
$\langle u[1 - s(\lambda)] \rangle_n = H[(1+z)^{1/2}] 
$ with $z:= a \lambda$ and therefore
\begin{equation}
\tilde{\Lambda}_{n,m} =\left.
\sum_{k=0}^n c_k 
\frac{{\rm d}^m}{{\rm d} z^m} (1 + z)^{k/2} \right|_{z=0}= \sum_{k=0}^n c_k \left(\frac{k}{2} \right)_m.
\label{lambda}
\end{equation}
Since $\Theta^n t^k = k^n t^k$ with $\Theta$ the Euler operator, it follows $(\Theta)_m t^k = (k)_m t^k$ and hence
$$\left( \frac{\Theta}{2}\right)_m H(t) = \sum_{k=0}^n c_k \left( \frac{\Theta}{2}\right)_m t^k.$$ Evaluating at 
$t=1,$ the coefficients $\tilde{\Lambda}_{n,m}$ 
in \eqref{lambda} have the form \eqref{firstcoeff}
with $\tilde{\Lambda}_{n,0} = 1.$ Applying the same reasoning to the denominator  $\langle1-2us(\lambda) \rangle_n$ yields 
\begin{equation}
\langle 1 - 2us(\lambda) \rangle_n =  \sum_{k\geq0} \Lambda_{n,k}\, \frac{a^k\lambda^k}{k!},
\label{den} 
\end{equation}
with coefficients $\{\Lambda_{n,k}\}$ given in \eqref{firstcoeff} and $\Lambda_{n,0} = 1.$ Using \eqref{ratio}, we obtain
\begin{equation*}
    \frac{\langle  u[1-s(\lambda)] \rangle_n}{\langle1-2us(\lambda) \rangle_n} = \sum_{m\geq0} M_{n,m} \frac{a^m\lambda^m}{m!} 
\end{equation*}
where $\{M_{n,m}\}$ are given in \eqref{ratioM}. Substituting into \eqref{ap:phiexpansion} gives
\begin{equation}
    \sum_{n \geq 0} \frac{\langle  u[1-s(\lambda)] \rangle_n}{\langle1-2us(\lambda) \rangle_n} \frac{(v y)^n}{n!} = 
     \sum_{n \geq 0} \left(\sum_{m\geq0} M_{n,m} \frac{a^m\lambda^m}{m!} \right) \frac{(v y)^n}{n!} = \sum_{k\geq0} l_k(y) \frac{\lambda^k}{k!},
\label{expT}
\end{equation}
with $\{l_k(y)\}$ in \eqref{0coeff}.  Assuming that $1-2u s(\lambda)\notin\{0,-1,-2,\dots\}$ in a neighbourhood of $\lambda=0$, the summations over $n$ and $m$ in \eqref{expT} can be interchanged. 
The result then follows by applying the Cauchy product to multiply the power
series in \eqref{expT} with
\begin{equation} 
y^{u(1-s(\lambda))} = y^u \exp \{ -u \log(y) s(\lambda) \} =  \sum_{k \geq 0} q_k \frac{\lambda^k}{k!}
\label{exp}
\end{equation}
 where $\{q_k\}$ are given in \eqref{0coeff}. 
\end{proof}
\begin{proof}[\bf Proof (Theorem 2)]
For large \(vy\), the Tricomi function in \eqref{GL} admits the
Poincar\'e asymptotic expansion
\begin{align}
&\Psi\left(
u[1-s(\lambda)],
1-2us(\lambda),
vy
\right)
\nonumber\\
&\quad =
\frac{1}{(vy)^{u[1-s(\lambda)]}}
\left\{
\sum_{n=0}^{N-1}
\langle u[1-s(\lambda)]\rangle_n
\langle u[1+s(\lambda)]\rangle_n
\left(-\frac{1}{vy}\right)^n
\frac{1}{n!}
+
\varepsilon_N(\lambda,vy)
\right\},
\label{ap:psiexpansion}
\end{align}
where, for fixed \(N\) and \(\lambda\) in a suitable neighborhood of
zero,
\(
\varepsilon_N(\lambda,vy)
=
O\left((vy)^{-N}\right),vy\to\infty.
\)
The associated formal asymptotic series is denoted by
\begin{equation}
{}_2F_0\left(
u[1-s(\lambda)],
u[1+s(\lambda)];
-\frac{1}{vy}
\right)=
\sum_{n\geq0}
\langle u[1-s(\lambda)]\rangle_n
\langle u[1+s(\lambda)]\rangle_n
\left(-\frac{1}{vy}\right)^n
\frac{1}{n!}.
\label{ap:2F0}
\end{equation}
Although the series in \eqref{ap:2F0} is generally divergent for
fixed \(vy\), it represents the Poincar\'e asymptotic expansion of the
Tricomi function and is used below as a formal series in \(1/(vy)\). 
The function $\langle u[1 - s(\lambda)] \rangle_n$ admits the expansion given in \eqref{num} with coefficients in \eqref{firstcoeff}. Proceeding analogously for $\langle u[1 + s(\lambda)] \rangle_n$, one obtains
\begin{equation}
\langle u[1 + s(\lambda)] \rangle_n =  \sum_{m\geq0} \bar{\Lambda}_{n,m} \frac{a^k\lambda^m}{m!}
\end{equation}
where $\bar{\Lambda}_{n,m}$ are given in \eqref{IIcoeffdowncomp} and $\bar{\Lambda}_{n,0} = 1.$ Taking the Cauchy product of the two series yields
\begin{equation}
\langle u[1 - s(\lambda)] \rangle_n \cdot  \langle u[1 + s(\lambda)] \rangle_n = \sum_{m\geq0} \bar{M}_{n,m} \frac{a^m\lambda^m}{m!}
\label{product}
\end{equation}
where $\{\bar M_{n,m}\}$ are given in \eqref{lbar}. Plugging \eqref{product} into \eqref{ap:psiexpansion} gives

\begin{eqnarray}
\!\!\!\!\!\!\! & & \!\!\!\!\!\!\! \!\!\!\!\!\!\! 
\!\!\!\!\!\!\! \frac{1}{(vy)^{u[1-s(\lambda)]}} \sum_{n \geq 0}^{\infty} {\langle  u[1-s(\lambda)] \rangle_n} \cdot{\langle u[1+s(\lambda)]  \rangle_n} \cdot \left(-\frac{1}{v y}\right)^n \cdot \frac{1}{n!} 
 \nonumber \\
  &\!\!\!\!\!\!\! \!\!\!\!\!\!\! \!\!\!\!\!\!\!\!\!\! = & \!\!\!\!\!\!\! \!\!\!\!\!\!\!\frac{1}{(vy)^{u[1-s(\lambda)]}} \sum_{n \geq 0}^{\infty} \left( \sum_{m\geq0} \bar{M}_{n,m} \frac{a^m\lambda^m}{m!} \right) \left(-\frac{1}{v y}\right)^n \cdot \frac{1}{n!}
= \frac{1}{(vy)^{u[1-s(\lambda)]}}\sum_{k\geq0} \bar{l}_k(y) \frac{\lambda^k}{k!}
\label{exp3}
\end{eqnarray}
with $\{\bar l_k(y)\}$ given in \eqref{lbar}. 
The final result then follows upon observing that $y^{,u[1-s(\lambda)]}$ simplifies when \eqref{exp3} is substituted into \eqref{GL}.
\\ \medskip \\
{\bf Acknowledgements}
\\  \\
The authors would like to thank Prof. Nuno M. Brites
(Lisbon School of Economics and Management, Universidade de Lisboa) for bringing this problem to our attention. E. Di Nardo is component of the GNAMPA-INDAM italian group. 

\end{proof}

\bibliographystyle{plainnat}
\bibliography{biblio}

@book{Allen2010,
  author    = {Allen, Linda J. S.},
  title     = {An Introduction to Stochastic Processes with Applications to Biology},
  publisher = {CRC Press},
  year      = {2010}
}

@book{Renshaw2011,
  author    = {Renshaw, Eric},
  title     = {Stochastic Population Processes},
  publisher = {Oxford University Press},
  year      = {2011}
}

@book{Mao2007,
  author    = {Mao, Xuerong},
  title     = {Stochastic Differential Equations and Applications},
  publisher = {Woodhead Publishing},
  year      = {2007}
}

@article{Beddington1977,
  author  = {Beddington, John R. and May, Robert M.},
  title   = {Harvesting natural populations in a randomly fluctuating environment},
  journal = {Science},
  volume  = {197},
  pages   = {463--465},
  year    = {1977}
}

@article{Lande1993,
  author  = {Lande, Russell and Engen, Steinar and S{\ae}ther, Bernt-Erik},
  title   = {Optimal harvesting of fluctuating populations with a risk of extinction},
  journal = {The American Naturalist},
  volume  = {142},
  pages   = {105--130},
  year    = {1993}
}

@article{Tuckwell1974,
  author  = {Tuckwell, H. C.},
  title   = {A study of some diffusion models of population growth},
  journal = {Theoretical Population Biology},
  volume  = {5},
  pages   = {345--357},
  year    = {1974}
}

@article{Braumann1999,
  author  = {Braumann, Carlos A.},
  title   = {Variable effort harvesting models in random environments},
  journal = {Mathematical Biosciences},
  volume  = {156},
  pages   = {1--19},
  year    = {1999}
}

@book{KarlinTaylor1981,
  author    = {Karlin, Samuel and Taylor, Howard M.},
  title     = {A Second Course in Stochastic Processes},
  publisher = {Academic Press},
  year      = {1981}
}

@article{BritesBraumann2022,
  author  = {Brites, Nuno M. and Braumann, Carlos A.},
  title   = {Moments and probability density of threshold crossing times for populations in random environments under sustainable harvesting policies},
  journal = {Computational Statistics},
  year    = {2022},
  doi     = {10.1007/s00180-022-01237-0}
}

@article{GietValloisWantz2015,
  author  = {Giet, Jean-S{\'e}bastien and Vallois, Pierre and Wantz-M{\'e}zi{\`e}res, Sandrine},
  title   = {The logistic S.D.E.},
  journal = {Theory of Stochastic Processes},
  volume  = {20},
  number  = {36},
  pages   = {28--62},
  year    = {2015}
}

@book{AbramowitzStegun1964,
  author    = {Abramowitz, Milton and Stegun, Irene A.},
  title     = {Handbook of Mathematical Functions with Formulas, Graphs, and Mathematical Tables},
  publisher = {National Bureau of Standards},
  series    = {Applied Mathematics Series},
  volume    = {55},
  year      = {1964},
  address   = {Washington, DC}
}

@article{DOM2024,
  author    = {Di Nardo, Elvira and D'Onofrio, Giuseppe and Martini, Tommaso},
  title     = {Orthogonal gamma–based expansion for the CIR’s first passage time distribution},
  journal   = {Applied Mathematics and Computation},
  volume    = {480},
  pages     = {128911},
  year      = {2024},
  doi       = {10.1016/j.amc.2024.128911}
}

@article{DiNardoDOnofrio2021,
  author    = {Di Nardo, Elvira and D'Onofrio, Giuseppe},
  title     = {On the cumulants of the first passage time of the inhomogeneous geometric Brownian motion},
  journal   = {Mathematics},
  volume    = {9},
  number    = {9},
  pages     = {956},
  year      = {2021},
  doi       = {10.3390/math9090956}
}

@article{AbateWhitt1996,
  author  = {Abate, Joseph and Whitt, Ward},
  title   = {An operational calculus for probability distributions via Laplace transforms},
  journal = {Advances in Applied Probability},
  volume  = {28},
  number  = {1},
  pages   = {75--113},
  year    = {1996}
}

@manual{fptdApproxManual,
  title        = {{fptdApprox}: Approximation of First-Passage-Time Densities for Diffusion Processes},
  author       = {Rom{\'a}n-Rom{\'a}n, Patricia and Serrano-P{\'e}rez, Juan J. and Torres-Ruiz, Francisco},
  year         = {2025},
  note         = {R package manual, version 2.5},
  url          = {https://cran.r-project.org/web/packages/fptdApprox/fptdApprox.pdf}
}

@book{charalambides2002enumerative,
  author    = {Charalambides, Charalambos A.},
  title     = {Enumerative Combinatorics},
  publisher = {Chapman \& Hall/CRC},
  year      = {2002},
  address   = {Boca Raton, FL},
  pages     = {609}
}

@book{jordan1949calculus,
  title={Calculus of Finite Differences},
  author={Jordan, Charles},
  edition={2},
  year={1949},
  publisher={Chelsea Publishing Company},
  address={New York}
}

@article{Otunuga2021,
  author  = {Otunuga, O. M.},
  title   = {Time-dependent probability density function for general stochastic logistic population model with harvesting},
  journal = {Physica A: Statistical Mechanics and its Applications},
  volume  = {573},
  pages   = {125931},
  year    = {2021},
  doi     = {10.1016/j.physa.2021.125931}
}

@article{Otunuga2025,
  author  = {Otunuga, O. M.},
  title   = {Stochastic modeling and first-passage-time analysis of oncological time metrics with dynamic tumor barriers},
  journal = {Scientific Reports},
  volume  = {15},
  pages   = {14941},
  year    = {2025},
  doi     = {10.1038/s41598-025-95475-z}
}

@book{Olver2010,
  author    = {Olver, Frank W. J. and Lozier, Daniel W. and Boisvert, Ronald F. and Clark, Charles W.},
  title     = {NIST Handbook of Mathematical Functions},
  publisher = {Cambridge University Press},
  address   = {New York},
  year      = {2010},
  note      = {National Institute of Standards and Technology, U.S. Department of Commerce}
}

@book{redner2001guide,
  title={A guide to first-passage processes},
  author={Redner, Sidney},
  year={2001},
  publisher={Cambridge university press}
}

@article{di2023approximating,
  title={Approximating the first passage time density from data using generalized Laguerre polynomials},
  author={Di Nardo, Elvira and D’Onofrio, Giuseppe and Martini, Tommaso},
  journal={Communications in Nonlinear Science and Numerical Simulation},
  volume={118},
  pages={106991},
  year={2023},
  publisher={Elsevier}
}

@article{BensoussanGlowinskiRascanu1992,
  author  = {Bensoussan, Alain and Glowinski, Roland and R{\u{a}}{\c{s}}canu, Adrian},
  title   = {Approximation of Some Stochastic Differential Equations by the Splitting Up Method},
  journal = {Applied Mathematics and Optimization},
  volume  = {25},
  number  = {2},
  pages   = {81--106},
  year    = {1992},
  doi     = {10.1007/BF01442569}
}

@article{Belisle1992,
  author  = {B\'elisle, C. J. P.},
  title   = {Convergence theorems for a class of simulated annealing algorithms on $\mathbb{R}^d$},
  journal = {Journal of Applied Probability},
  volume  = {29},
  number  = {4},
  pages   = {885--895},
  year    = {1992}
}

@phdthesis{Martini2024,
  author       = {Tommaso Martini},
  title        = {Statistical and Probabilistic Approaches to Hydrological Data Analysis: Rainfall Patterns, Copula-like Models and First Passage Time Approximations},
  school       = {Università degli Studi di Torino},
  year         = {2024},
  month        = dec,
  type         = {PhD thesis},
  url          = {https://iris.unito.it/handle/2318/2043022}
}

\end{document}